\documentclass{article}
\usepackage[a4paper, total={6in, 8in}]{geometry}
\usepackage{amsmath}
\usepackage{amssymb}
\usepackage{amsthm}
\usepackage{multirow}
\usepackage{booktabs}
\usepackage{hyperref}
\usepackage{array}
\usepackage[ruled,vlined,linesnumbered]{algorithm2e}
\usepackage[all,pdf]{xy}
\usepackage{bm}
\usepackage{extarrows}

\usepackage{graphicx}
\graphicspath{{./}{./fig/}}
\usepackage{subfigure}

\usepackage{xcolor}
\DeclareMathOperator*{\argmin}{argmin}
\DeclareMathOperator*{\card}{card}

\newtheorem{theorem}{Theorem}

\newtheorem{lemma}{Lemma}

\newtheorem{assumption}{Assumption}
\newtheorem{remark}{Remark}
\newtheorem{example}{Example}

\usepackage{diagbox} 

\renewcommand{\bf}{\pmb}
\newcommand{\bx}{\boldsymbol{x}}
\newcommand{\nrm}[1]{\left\Vert {#1} \right\Vert}

\newcommand{\Ebb}{\mathbb{E}}

\newcommand{\Abf}{\mathbf{A}}
\newcommand{\bbf}{\mathbf{b}}

\newcommand{\Rbf}{\mathbf{R}}

\usepackage{authblk}

\title{Weak Collocation Regression method: fast reveal hidden stochastic dynamics from high-dimensional aggregate data}

\author[1]{Liwei Lu}
\author[2]{Zhijun Zeng}
\author[3]{Yan Jiang}
\author[4]{Yi Zhu\thanks{ yizhu@tsinghua.edu.cn}}
\author[5]{Pipi Hu\thanks{Corresponding author, pisquare@microsoft.com}}

\affil[1,2,3,4]{Yau Mathematical Sciences Center, Tsinghua University, Beijing, 100084, China.}
\affil[4]{Yanqi Lake Beijing Institute of Mathematical Sciences and Applications, Beijing, 101408, China.}
\affil[5]{Microsoft Research AI4Science, Beijing, 100080, China.}

\date{\today}

\begin{document}
\maketitle

\begin{abstract}
Revealing hidden dynamics from the stochastic data is a challenging problem as the randomness takes part in the evolution of the data.  The problem becomes exceedingly hard if the trajectories of the stochastic data are absent in many scenarios. In this work, we propose the Weak Collocation Regression (WCR) method to learn the dynamics from the stochastic data without the labels of trajectories. This method utilize  the  governing equation of the probability distribution function--the Fokker-Planck (FP) equation. Using its weak form and integration by parts, we move all the spacial derivatives of the distribution function to the test functions which can be computed explicitly. Since the data is a sampling of the corresponding distribution function, we can compute the integrations in the weak form, which has no spacial derivatives on the distribution functions, by simply adding the values of the integrands 
at the data points. We further assume the unknown drift and diffusion terms can be expanded by the base functions in a dictionary with the coefficients to be determined. Cooperating  the collocation treatment and linear multi-step methods, we transfer the revealing process to a linear algebraic system. Using the sparse regression, we eventually obtain the unknown coefficients and hence the hidden stochastic dynamics. The numerical experiments show that our method is flexible and fast, which reveals the dynamics within seconds in the multi-dimensional problems and can be extended to high dimensional data. The complex tasks with variable-dependent diffusion and coupled drift can be correctly identified by WCR and the performance is robust, achieving high accuracy in the cases of noisy data. The rigorous error estimate is also included to support our various numerical experiments. 


\end{abstract}


\section{Introduction}

Nowadays, a large amount of data has been collected in different realms, and revealing the hidden dynamics buried in the data is an essential topic in the scientific discovery and engineering applications. On one hand, the studies in the past several centuries have proven the success of the differential equations derived from the so-called first principle in the descriptions of the phenomena of nature, such as the Navier-Stokes equation in hydrodynamics for fluid dynamics \cite{temam2001navier}, Schr\"{o}dinger equation in quantum mechanics for probability current \cite{ballentine2014quantum}, Black-Scholes in computational finance for option pricing \cite{klein1996pricing}. 
On the other hand, machine learning, especially deep learning, in recent years has attained tremendous success in computer vision, natural language processing, and many other topics in computer science \cite{goodfellow2016deep}. 
Leveraging the structure of differential equations and machine learning in the data analysis shows its prospective performance in many studies \cite{karniadakis2021physics, raissi2020hidden, brunton2022data, hu2022revealing}. One of the most important focus is pouring new structures onto modeling the hidden dynamics from data. The physical informed neural network (PINN) adds physical constraints to the data by adding the residual of the differential equations to the loss and making learning coefficients of the unknown terms of the governing equations reliable \cite{raissi2019physics}. Brunton \textit{et al.} proposed a framework named ``SINDy'' by combining regression and sparse identification to reveal nonlinear dynamical systems \cite{kutz16sparse}. The time-series data always contain a lot of missing points and even flaws with high noisy level, making analysis hard and tricky. Hu \textit{et al.} proposed using symbolic ODE (ordinary differential equations) to reveal hidden dynamics from time series data \cite{hu2022revealing} with the integral form making the learning process of the data with large time step more stable. More related works about inferring differential equations, see \cite{long2018pde,long2019pde, both2021deepmod, chen2021physics, lin2021binet, lou2021physics, lu2019deeponet}.   
Most of these works leverage the structure of differential equations and machine learning schemes to model the data in the deterministic system. However, stochastic dynamics plays a vital role in the applications such as the Brownian motion in the molecular modeling and the financial predictions \cite{dufresne2001integral}. In contrast to the ``noise" of data generated from the measurement, the randomness in the stochastic dynamics contributes the evolution of the process, making it more complicated for modeling. Fortunately, with subtly designed structure, it is still possible to reveal the hidden dynamics from the stochastic data.

In the realm of stochastic dynamics, aggregate data refers to a data format in which the full trajectory of each individual modeled by the evolution of state is not available, but rather a sample from the distribution of state at a certain time point is available \cite{ma2021learning}. For example, the data collected for single-cell DNA sequence analysis, bird migration, and social gathering are aggregate data as the individual trajectories for a long time are not possible to follow with only the collection of the states of different individuals obtained. In contrast, trajectory data includes all the information of the individual data along the time. Thus, in some literature, trajectory data and aggregate data are also called ``paired" and ``unpaired" data respectively \cite{yang2022generative}.

 For trajectory data, there exist many methods developed such as Hamiltonian neural networks \cite{greydanus2019hamiltonian}, Hidden Markov Model (HMM) \cite{alshamaa2019hidden}, Kalman Filter (KF) \cite{khalkhali2020vehicle}, Particle Filter (PF) \cite{santos2019unmanned} and related works \cite{ma2021learning, fang2019road}. Revealing determined hidden dynamics from data above can also be regarded as one special case of trajectory data when stochastic term vanishes. 

However, for aggregate data, few works are investigated because of the absence of individual trajectories. By leveraging the Fokker-Planck equation, the governing function of the probability density function of the Winner process variable, researchers investigate the revealing hidden stochastic dynamics from aggregate data. Zhou \textit{et al.} propose a novel method using the weak form of the Fokker Planck Equation (FPE), a partial differential equation, to describe the density evolution of data in a sampled form, which is then combined with the Wasserstein generative adversarial network (WGAN) in the training process \cite{ma2021learning}.
Chen \textit{et al.} proposed a method that leverages the physical-informed structure of Fokker-Planck equations and the approximation of the probability density function and reveals the hidden dynamics from sampling data at several time points \cite{chen2021solving}. Yang \textit{et al.} \cite{yang2022generative} integrated the distance measure such as Wasserstein distance used in WGAN \cite{arjovsky2017wasserstein} and the forward numerical solution of the parameterized stochastic differential equations as the Physics-Informed Deep Generative Models to reveal the hidden dynamics of the data.

In this work, by leveraging the weak form of the Fokker-Planck equation, we proposed a Weak Collocation Regression (WCR) method to reveal the hidden stochastic dynamics from nonequally-spaced temporal aggregate data. The Fokker-Planck equation describes the time evolution of the probability density function of the random variable in the Brownian motion. By the weak form of the Fokker-Planck equation, one can reduce the aggregate data at different time points to one dimensional temporal sequence where we have used the Monte-Carlo summation for the approximation of the weak form using data. Temporal derivatives are reduced by Linear Multistep Methods (LMMs), and then the linear system constructed by collocation of the kernels gives a precise approximation of the stochastic dynamics. The benefits are three folds.

\begin{enumerate}
    \item Remarkable performance. Our method has a low computational cost, comparable accuracy, controlled error, and the dimensional curse can be lessened. Numerical experiments show that the 1-dimensional problem can be easily revealed within $0.02$ second on the MacBook, while the computational time of the 3 or 4-dimensional problem can be limited to seconds. We can significantly reduce the error by changing the time interval and sample number. Directly numerically solving the inverse problem of stochastic differential equations may encounter the curse of dimensionality \cite{bellman2015adaptive}, where the tendency of numerical techniques requires a high computational cost growing exponentially with respect to the dimension of the variables. Benefits from the Monte-Carlo summation and the random collocations methods, the curse of dimensionality is lessened. Further, our method has natural potential for parallel computation.
    \item Robustness in the data with different qualities. Our method can handle a small amount of non-equally noise data without trajectory information and even partially data missing. No trajectories are needed; the only thing needed is the summation over all the points at each given time. Our method hence has natural permutation invariance with respect to the observations of different individuals at the same time snapshot. Our method can be applied to different number of individuals at different time snapshots. The obstacle of the measurement error can be reduced greatly since the summation of the aggregate data is resistant to white noise. Numerical experiments illustrate the stability of our methods. 
    \item Complex tasks. Our method is suited for the coupled drift term derived from Sombrero potential, variable-dependent diffusion term, and high-contrast problem and can be extended to a wider range. Numerical experiments show that it obtained consistent results if we expand each drift and diffusion term in a high-order polynomial. It shows that in dealing with the complex task above, our method has a wider representation and shows superior performance in revealing the hidden dynamics with a mild amount of data in a high accuracy.
\end{enumerate}

\section{Methodology}

In this paper, we consider the scenarios of the so-called aggregate data (unpaired data), where the trajectory information of each individual modeled by the evolution of state is unavailable. Only the collection of the $N_{t_j}$ samples $\mathbb{X}_{t_j}=\{\bx^i_{t_j}\}_{i=1}^{N_{t_j}}$ from the distribution of state $X_t$ at certain time point $t_j, j=1,2,\cdots, L$ is available. In this work, we call $\mathbb{X}_{t_j}$ one snapshot and hence there are $L$ snapshots in the data set with respect to $L$ time points.

For simplicity, we denote the set of the aggregate data as 
\begin{equation}
\mathbb{X} \overset{\Delta}{=} \{\mathbb{X}_{t_j}\}_{j=1}^{L}=\big\{ \{\bx^i_{t_j}\}_{i=1}^{N_{t_j}}\big\}_{j=1}^{L},
\end{equation}
 where $t_j$ is the time index of the $j$-th snapshot in the total $L$ snapshots, $\bf{x}_{t_j}^{i}$ is the $i$-th sample of the variable $X_t$ at time $t_j$.
To model the stochastic dynamics of the data, the general stochastic differential equations (SDEs) for Brownian motion are considered
\begin{equation}\label{eq.sde}
dX_t = \bf{\mu}(X_t,t)dt + \bf{\sigma}(X_t, t)dW_t,
\end{equation}
with drift term $\bf{\mu}(X_t,t)\in\mathbb{R}^d$ and diffusion term $\bf{\sigma}(X_t, t)\in\mathbb{R}^{d\times w}$ . Here, $X_t$ is the $d$-dimensional random variable of the data, and $W_t$ is the $w$-dimensional standard Brownian motion. 

The density function $p(x,t)$ of the above random variable $X_t$ can be described by the Fokker-Planck equation (FP), see \cite{risken1996fokker}, and we restate the result in Lemma \ref{lemma.fp}.
\begin{lemma}\label{lemma.fp}
Suppose $\{X_t\}$ solves the SDEs \eqref{eq.sde}, then the probability density function $p(x,t)$ of the random variable $X_t$ satisfies the following d-dimensional Fokker-Planck equation by the It\^{o} integral
\begin{equation}\label{eq.fp}
\partial_t p = -\nabla\cdot(\bf{\mu} p) + \sum_{i,j}^d \partial_{ij}(D_{ij}p),
\end{equation}
where $\bx\in\mathbb{R}^d$, $t\in[0,T]\subset\mathbb{R}$, $p = p(\bx,t)\in\mathbb{R}$ is the probability density function with $\int_{\mathbb{R}^d} p(\bx,t)dx=1$, $\boldsymbol{\mu} = [\mu_1(\bx,t), \mu_2(\bx,t), \cdots, \mu_d(\bx,t)]^T$, and the diffusion matrix $[D_{ij}]=[D_{ij}(\bx,t)]$ is given by 
\begin{equation}
D = \frac{1}{2}\bf{\sigma}\bf{\sigma}^T.
\end{equation}
\end{lemma}

The Fokker-Planck equation bridges the gap between stochastic dynamics and the distributions of the data samples by sharing drift term and diffusion relation in SDE \eqref{eq.sde} and FP equation \eqref{eq.fp}. To reveal the hidden stochastic dynamics \eqref{eq.sde}, equivalently one can reveal the unknown terms in FP equation \eqref{eq.fp} by the samples which follow the corresponding probability distribution.

However, there is a huge gap between the density function and the data samples. Directly modeling the data by the Fokker-Planck equation usually needs the temporal and spatial derivatives of the density function $p(t, x)$ with respect to $t$ and $x$,  requiring a large amount of samples for constructing a smooth density function.  It becomes even impossible to construct the smooth density function when the dimension increases. To overcome this difficulty, in this work, we propose a framework using the weak form with collocation integral kernels instead of reconstructing density function to fast reveal hidden stochastic dynamics with a mild amount of data. We briefly introduce the methods here. 
We first write the weak form of the FP equation and then integrated by parts, the partial derivatives in the weak form are moved to the explicit kernel function, making the computation of the spatial derivatives much easier. Followed by the Monte-Carlo summation, the integration over space can be given by the summation over the samples. And the terms with temporal derivatives can be approximated by the linear multi-step method (LMMs). Finally, with a basis expansion of the unknown drift and diffusion, the linear system is built, and the sparse regression gives a good approximation of the unknown terms. Figure \ref{fig.scheme} depicts the picture of the methodology. The weak form, LMMs, and the regression model would be detailedly discussed in the next several subsections.

\subsection{Leverage the weak form of the Fokker-Planck equation}
In the realm of PDE theories, weak solutions from the weak form of PDEs are vital in analysis and applications. The weak form is introduced to solving PDEs with neural networks by the so-called WAN method \cite{zang2020weak} and then extended to revealing the unknown parameters in PDEs  \cite{bao2020numerical}. Borrowing the ideas from the Galerkin methods, we alter the weak form of the FP equations by using collocations of the kernel function as test functions. The space of the test functions thus are approximated by collocations of the kernels unlike the maximum optimization for the tunable parameters in WAN. Without maximum steps for the test function, the optimization only for the unknown terms in the revealing stochastic dynamics achieves stable performance with high robustness.

The test function is given as the kernel function $\psi(x):\mathbb{R}^d\to \mathbb{R}$. For example, a typical choice can be the Gaussian function with the form 
\begin{equation}\label{eq.gauss}
\phi(\bx, \bf{\rho}, \bf{\gamma}) = \Pi_{i=1}^d \frac{1}{\gamma_i \sqrt{2 \pi}} e^{-\frac{1}{2}\left(\frac{x_i - \rho_i}{\gamma_i}\right)^{2}},
\end{equation}
where $\bx=(x_1,\ldots, x_d)^T\in\mathbb{R}^d$,  $\bf{\rho}=(\rho_1,\ldots, \rho_d)^T$ and $\bf{\gamma}=(\gamma_1,\ldots, \gamma_d)^T$ are the expectation and  standard deviation of the Gaussian function.

Multiply $\psi(\bx)$ on both sides of equation \eqref{eq.fp}, integrate by parts, and one obtains
\begin{equation}\label{eq.integral}
\begin{split}
\frac{d}{dt}\int_{\mathbb{R}^d} p(\bx,t)\psi(\bx)d\bx
= \int_{\mathbb{R}^d} p(\bx,t)\bf{\mu}(\bx,t)\cdot \nabla \psi(\bx)d\bx +\int_{\mathbb{R}^d} p(\bx,t)\sum_{rs}^dD_{sr}\partial_{rs}\psi(\bx)d\bx,
\end{split}
\end{equation}
where $\partial_{rs} = \frac{\partial^2}{\partial x_r\partial x_s}$ and $r, s$ are indices of the dimension. 

As $p(\bx,t)$ is the density of the variable $X_t$ at time $t$ with $\int p(\bx,t)d\bx=1$, we can rewrite equation \eqref{eq.integral} as the expectation form
\begin{equation}
\frac{d}{dt}\mathbb{E}_{\bx\sim p(\bx, t)}[\psi(\bx)] = \mathbb{E}_{\bx\sim p(\bx, t)}[\bf{\mu}(\bx,t)\cdot\nabla\psi(\bx)]+\mathbb{E}_{\bx\sim p(\bx, t)}[\sum_{rs}^dD_{sr}\partial_{rs}\psi(\bx)],
\end{equation}
where $\mathbb{E}_{\bx\sim p(\bx,t)}$ is the expectation over the probability density function $p(\bx,t)$.  

\subsection{Approximate the weak form using data}
In the scenarios of real applications, the probability distribution $p(\bx,t)$ is hardly to obtain but the data of many samplings from the distribution is always feasible thanks to the modern technology and instruments. 
By the law of large numbers, the expectation over the distribution can be approximated by the summation over the samplings of the variable related to the distribution, i.e.,
\begin{equation}
 \frac{1}{N_t}\sum_{i=1}^{N_t}\psi(\bx_t^i) \sim \mathbb{E}_{\bx\sim p(\bx, t)}[\psi(\bx)],
\end{equation}
where $N_t$ is the number of the samplings  of the variable $X_t$ at time $t$, and $\bx^i_t$ is the $i$-th sampling.

 For simplicity, we denote the data set over time as $\mathbb{X}_t\overset{\Delta}{=}\{\bx^i_{t}\}_{i=1}^{N_t}$.
Thus, with the data set of the samplings $\mathbb{X}_t$, the weak form of the FP equation \eqref{eq.fp} yields 
\begin{equation}\label{eq.monte}
\frac{d}{dt}\biggl(\frac{1}{N_t}\sum_{i=1}^{N_t}\psi(\bx^i_{t})\biggr)
 = \frac{1}{N_t}\sum_{i=1}^{N_t}\bf{\mu}(\bx_t^i,t)\cdot\nabla\psi(\bx^i_{t})+\frac{1}{N_t}\sum_{i=1}^{N_t}(\sum_{rs}^dD_{sr}\partial_{rs}\psi(\bx^i_{t})) + \epsilon,
\end{equation}
where we have used the approximations of the expectations by
the summations over the samplings, and the discussion of the error term $\epsilon$ can be found in Appendix \ref{app.error}.
The information of the drift and diffusion terms are hence related with the data set $\mathbb{X}_t$ in equation \eqref{eq.monte}. The dimension of the data has been reduced to one by the summation over the samplings with only the dependence on the time $t$ left. This form reveals that our method is naturally suited for the high dimensional problem.

For simplicity, we denote 
\begin{equation}\label{eq.f}
\frac{d}{dt} y(\mathbb{X}_t) 
 = f(\mathbb{X}_t, \bf{\mu}, D),
\end{equation}
where 
$$y(\mathbb{X}_t)=\frac{1}{N_t}\sum_{i=1}^{N_t}\psi(\bx^i_{t})$$ and $$f(\mathbb{X}_t, \bf{\mu}, D)=\frac{1}{N_t}\sum_{i=1}^{N_t}\bf{\mu}(\bx_t^i,t)\cdot\nabla\psi(\bx^i_{t})+\frac{1}{N_t}\sum_{i=1}^{N_t}(\sum_{rs}^dD_{sr}\partial_{rs}\psi(\bx^i_{t}))$$ are scalar functions varying over time $t$.
 
\subsection{Approximate the temporal derivatives}

The equation \eqref{eq.f} is indeed a 1-dimensional parameterized ordinary differential equation where the solutions on the discrete times were given with $y(\mathbb{X})$. The problem of finding unknown terms in FP equation now reduces to the inverse problem of the 1-dimensional ordinary differential equation \eqref{eq.f}. Many methods recently have been developed to reveal the dynamics such as SINDy \cite{rudy2017data}, PINN \cite{raissi2019physics} and ODENet \cite{hu2022revealing}. Considering the simplicity, in this work, we directly use the implicit form of Linear Multistep Methods (LMMs) to construct an explicit discrete form of equation \eqref{eq.f}.

Linear Multistep Methods (LMMs) have been developed as popular numerical schemes for the integration of the ordinary differential equations for the known dynamic systems \cite{goldstine2012history} with well-established mathematical theories \cite{atkinson2011numerical}. Raissi \textit{et al.} constructed multistep neural networks for data-driven discovery of nonlinear dynamical systems leveraging LMMs \cite{raissi2018multistep}.
Recently Du \textit{et al.}  applied LMMs in learning hidden dynamics from the data of given state with theoretical analysis of the stabilities and the convergence for the inverse problem \cite{keller2021discovery, du2021discovery}. In their work, equally-spaced version of LMMs is used. However, real data in the record usually contain a lot of missing points and even flaws, making equally-spaced version of LMMs hard to apply. In this work, with the help of variable step size Adams methods \cite{norsett1987solving}, we can deal with the non-equally spaced temporal data for the discovery of the hidden dynamics. 

For equally-spaced temporal data, Implicit Adams methods of trapezoidal rule, Milne method, 2-step backward differentiation method (BDF2) and Adams-Moulton methods are used in this work and some of them are listed below

\begin{equation} \label{eq.trapezoidal}
\text{Trapezoidal rule:} \quad   y_{n+1} - y_{n} = \frac{h}{2}(f_{n+1} + f_n),
\end{equation}
\begin{equation}\label{eq.am}
    \text{Milne method:} \quad y_{n+1}-y_{n-1} = \frac{h}{3}(f_{n+1} + 4f_{n} + f_{n-1}),
\end{equation}
where $y_n = y(t_n)$, $f_n = f(t_n, \theta)$ and $h\equiv h_n=t_{n+1}-t_n$ for $n=1,2,\cdots, L-1$.

For more general temporal data without equally spaced time, the 2-step BDF-formula of variable step size methods is as follows
\begin{equation}
y_{n+1}-\frac{\left(1+\omega_{n}\right)^{2}}{1+2 \omega_{n}} y_{n}+\frac{\omega_{n}^{2}}{1+2 \omega_{n}} y_{n-1}=h_{n} \frac{1+\omega_{n}}{1+2 \omega_{n}} f_{n+1},
\end{equation}
where $\omega_n = h_n/h_{n-1}$.

Recall that the Implict Adams methods of trapezoidal rule requires only two adjacent time points, thus the variable step-size version reads
\begin{equation} \label{eq.nonequalLMM2}
y_{n+1} - y_{n} = \frac{h_n}{2}(f_{n+1} + f_n).
\end{equation}

Our numerical experiments show that the trapezoidal rule performs well when only three time snapshots are avaliable. In the contrast, the LMMs methods requiring 2 steps fails in such scenario even if they have higher order of accuracy. Other types of the numerical schemes have been also tested but with a worse performance compared with the LMMs.

By LMMs, we derive the discrete form of equation \eqref{eq.f} as the following form
\begin{equation}\label{eq.fd}
\bf{\hat{y}}(\mathbb{X}) = \bf{\hat{f}}(\mathbb{X},\bf{\mu}, D),
\end{equation}
where $\bf{\hat{y}}(\mathbb{X})$ is a vector constructed by $\bf{y}(\mathbb{X})=\left\{y(\mathbb{X}_{t_j})\right\}_{j=1}^L$, and  $\bf{\hat{f}}(\mathbb{X},\bf{\mu}, D)$ is a vector with the same size given by $\bf{f}(\mathbb{X},\bf{\mu}, D)=\left\{f(\mathbb{X}_{t_j},\bf{\mu}, D)\right\}_{j=1}^L$. Now the equation \eqref{eq.fd} gives us an algebraic equation, from which, by an apt ansatz form for the drift and diffusion terms, we can reveal $\bf{\mu}$ and $D$ with the data set $\mathbb{X}$.  

\subsection{Build the regression model}
In this subsection, we investigate the sparse regression and collocations of the kernels to solve $\bf{\mu}$ and $D$ from \eqref{eq.fd} with the data set $\mathbb{X}$. The dictionary representations are adopted for the unknown terms. One typical choice can be the polynomial basis for the approximation of the drift and diffusion terms with the coefficients of the basis to be determined. Let 
\begin{equation}
\Lambda = \{1, x_1, x_2, \cdots, x_d, x_1x_1, x_1x_2, \cdots, x_d^p\}^T
\end{equation}
denote the $p$-th order complete polynomials with respect to the variable $\bf{x}=(x_1,\cdots, x_d)^T$. The number of the terms of $\Lambda$ is $b \overset{\Delta}{=}|\Lambda| = \begin{pmatrix}
p+d\\
p
\end{pmatrix}$. In real applications, if we have some knowledge about the stochastic process, more flexible basis set $\Lambda$ can be chosen with diverse and fewer candidate basis.

Here, for simplicity, we suppose that $\bf{\mu}$ and $D$ are independent of $t$. And the components $\mu_i$ of $\bf{\mu}=[\mu_i]$ and $D_{ij}$ of $D=[D_{ij}]$ expand as 
\begin{equation}\label{eq.expan}
\mu_i = \sum_{j=1}^{b} \zeta_{ij}^\mu \Lambda_j, 
\quad\text{and } D_{ij} = \sum_{k=1}^{b} \zeta_{ijk}^D \Lambda_k, 
\end{equation}
where $\Lambda_j$ is the $j$-th component in the basis set $\Lambda$. 

Now, equation \eqref{eq.f} has the following matrix form
\begin{equation} \label{eq.B}
    \frac{d}{dt} \bf{y}(\mathbb{X}) = B(\mathbb{X}) \bf{\zeta}
\end{equation}
where
\begin{equation} \label{eq.zeta}
    \bf{\zeta} = \{ \underbrace{\zeta_{11}^\mu, \cdots, \zeta_{1b}^\mu, \cdots, \zeta_{d1}^\mu, \cdots, \zeta_{db}^\mu}_{\text{flatten of }\bf{\zeta}^\mu}, \underbrace{\zeta_{111}^D, \cdots, \zeta_{11b}^D, \cdots, \zeta_{dd1}^D, \cdots, \zeta_{ddb}^D}_{\text{flatten of }\bf{\zeta}^D}\}^T,
\end{equation}
and the flatten scheme of $\bf{\zeta}^\mu$ and $\bf{\zeta}^D$ is shown as 
\begin{equation*}
    \bf{\zeta}^\mu:
    \begin{pmatrix}
    \xymatrix@R=2ex@C=2ex{
    \zeta_{11}^\mu & \zeta_{12}^\mu & \cdots & \zeta_{1b}^\mu \\
    \zeta_{21}^\mu & \zeta_{22}^\mu & \cdots & \zeta_{2b}^\mu \\
    \cdots & \cdots & \cdots & \cdots \\
    \zeta_{d1}^\mu & \zeta_{d2}^\mu & \cdots & \zeta_{db}^\mu 
    \ar@{.>}"1,1";"1,4"
    \ar@{.>}"2,1";"2,4"
    \ar@{.>}"4,1";"4,4"
    \ar@{.>}"1,4";"2,1"
    \ar@{.>}"2,4";"3,1"
    \ar@{.>}"3,4";"4,1"
    } \\
    \end{pmatrix},
    \quad  \bf{\zeta}^D:
    \begin{pmatrix}
    \xymatrix@R=1.5ex@C=1.5ex{
    & & \zeta_{11b}^D & & & \zeta_{d1b}^D \\
    & \zeta_{112}^D \ar@{-}[ddd]|!{[d];[d]}\hole & & & \zeta_{d12}^D & \\
    \zeta_{111}^D & & & \zeta_{d11}^D & & \\
    & & \zeta_{1db}^D & & & \zeta_{ddb}^D \\
    & \zeta_{1d2}^D & & & \zeta_{dd2}^D \ar@{-}[lll]|!{[l];[l]}\hole & \\
    \zeta_{1d1}^D & & & \zeta_{dd1}^D & &
    \ar@{-}"3,1";"6,1" \ar@{-}"6,1";"6,4"
    \ar@{-}"3,1";"3,4" \ar@{-}"3,4";"6,4"
    \ar@{-}"2,2";"2,5" \ar@{-}"2,5";"5,5"
    \ar@{-}"1,3";"1,6" \ar@{-}"1,6";"4,6"
    \ar@{-}"1,3";"2,3" \ar@{-}"2,3";"3,3" \ar@{-}"3,3";"4,3"
    \ar@{-}"4,3";"4,4" \ar@{-}"4,4";"4,5" \ar@{-}"4,5";"4,6"
    \ar@{.>}"3,1";"1,3" \ar@{.>}"6,1";"4,3"
    \ar@{.>}"3,4";"1,6" \ar@{.>}"6,4";"4,6"
    \ar@{.>}"1,3";"4,1" \ar@{.>}"3,3";"6,1"
    \ar@{.>}"1,6";"4,4" \ar@{.>}"3,6";"6,4"
    \ar@{.>}"4,1";"2,3" \ar@{.>}"4,4";"2,6"
    }
    \end{pmatrix}
\end{equation*}
and $B(\mathbb{X}):=\begin{pmatrix}
    \bf{b}_1^T(\mathbb{X})\\
    \vdots\\\bf{b}_{L}^T(\mathbb{X})\end{pmatrix}$ is a known coefficient matrix with size $L\times(db+d^2b)$, $\bf{y}(\mathbb{X}):=\begin{pmatrix}
    y_1(\mathbb{X})\\ \vdots\\y_{L}(\mathbb{X})\end{pmatrix}$ is a colomn vector with size $L$.

By applying linear multistep methods (LMMs) on the temporal derivative in \eqref{eq.B}, one can obtain a linear system about the coefficient vector $\bf{\zeta}$ as
\begin{equation}\label{eq.linear}
A(\mathbb{X}) \bf{\zeta} = \bf{\hat{y}}(\mathbb{X}),
\end{equation}
where $A(\mathbb{X})$ and $\bf{\hat{y}}(\mathbb{X})$ are constructed by $B(\mathbb{X})$ and $\bf{y}(\mathbb{X})$ in \eqref{eq.B} because of the linearity of the system, and the coefficient vector $\bf{\zeta}$ are collected and vectorized from all the coefficients $\left\{\left\{\zeta_{ij}^\mu\right\}_{j=1}^{b}\right\}_{i=1}^{d}$ and $\left\{\left\{\zeta_{ijk}^D\right\}_{k=1}^{b}\right\}_{i,j=1}^{d}$ in the expansions of $\bf{\mu}$ and $D$ shown in \eqref{eq.zeta}. Now, equation \eqref{eq.linear} gives the relation between the data set $\mathbb{X}$ and the unknown parameters in the polynomial expansions of the hidden dynamics. 

\begin{example}
To make it clear, we give an example to illustrate how the linear system  is obtained by applying LMMs such as the trapezoidal rule \eqref{eq.trapezoidal} on the equation \eqref{eq.B} with temporal derivatives. 

In equation \eqref{eq.B}, we set $ n=db+d^2b $ and
\begin{equation}
    B(\mathbb{X}) = \begin{pmatrix}
        b_{11} & b_{12} & \cdots & b_{1n} \\ 
        b_{21} & b_{22} & \cdots & b_{2n} \\ 
        \vdots & \vdots & \ddots & \vdots \\
        b_{L1} & b_{L2} & \cdots & b_{Ln} \\ 
    \end{pmatrix}, \quad
    \bf{y}(\mathbb{X}) = \begin{pmatrix}
        y_1(\mathbb{X}) \\ y_2(\mathbb{X}) \\ \vdots \\ y_L(\mathbb{X}) \\
    \end{pmatrix}
\end{equation}
By the trapezoidal rule, the matrix $A(\mathbb{X})\in\mathbb{R}^{(L-1)\times(db+d^2b)}$ and vector $\bf{\hat{y}}(\mathbb{X})\in\mathbb{R}^{L-1}$ in the linear system can be assembled by 
\begin{small}
\begin{equation}
    A(\mathbb{X}) = \frac{h}{2}\begin{pmatrix}
        b_{11}+b_{21} & b_{12}+b_{22} & \cdots & b_{1n}+b_{2n} \\ 
        b_{21}+b_{31} & b_{22}+b_{32} & \cdots & b_{2n}+b_{3n} \\ 
        \vdots & \vdots & \ddots & \vdots \\
        b_{L-1,1}+b_{L1} & b_{L-1,2}+b_{L2} & \cdots & b_{L-1,n}+b_{Ln} \\ 
    \end{pmatrix}, \quad
    \bf{\hat{y}}(\mathbb{X}) = \begin{pmatrix}
        y_2(\mathbb{X})-y_1(\mathbb{X}) \\ y_3(\mathbb{X})-y_2(\mathbb{X}) \\ \vdots \\ y_L(\mathbb{X})-y_{L-1}(\mathbb{X}) \\
    \end{pmatrix}.
\end{equation}
\end{small}

Here, the rows of the matrix $A(\mathbb{X})$ or the elements of the vector $\bf{\hat{y}}(\mathbb{X})$ are constructed by the rows of $B(\mathbb{X})$ or the elements of $\bf{y}(\mathbb{X})$ where the number of rows is reduced by one because of the trapezoidal rule \eqref{eq.trapezoidal}.
\end{example}


A sparse regression applied to the linear system then reveals the hidden dynamics from the data. But before that, we would discuss the collocation strategies to improve the robustness and accuracy. Recall that in this work, the test function $\psi$ is considered as the Gaussian function  $\psi=\phi(\bx, \bf{\rho}, \bf{\gamma}) \overset{\Delta}{=} \Pi_{i=1}^d \frac{1}{\gamma_i \sqrt{2 \pi}} e^{-\frac{1}{2}\left(\frac{x_i - \rho_i}{\gamma_i}\right)^{2}}$ given by \eqref{eq.gauss},  one can easily obtain the 
the specific form of the linear system \eqref{eq.linear} as
\begin{equation}\label{eq.gl}
A(\mathbb{X}, \bf{\rho}, \bf{\gamma}) \bf{\zeta} = \bf{\hat{y}}(\mathbb{X}, \bf{\rho}, \bf{\gamma})
\end{equation}
 by the replacement of $\psi(\bf{x})=\phi(\bf{x}, \bf{\rho}, \bf{\gamma})$.

The collections of the test functions is taken as $\mathbb{C}^d_{\bf{\rho}, \bf{\gamma}}=\left\{\phi(\cdot, \bf{\rho}_m, \bf{\gamma}_m) \right\}_{m=1}^M$ where $\bf{\rho}_m\in \mathbb{R}^d$ and $\pmb{\gamma}_m \in \mathbb{R}^d$. Here we have used the notations $\bf{\rho}_m$ and $\bf{\gamma}_m$ to denote the expectation and standard deviation of the $m$-th Gaussian function.
Taking the $m$-th Gaussian function $\phi(\cdot, \bf{\rho}_m, \bf{\gamma}_m)$ as the test function, 
the linear system with respect to $\bf{\zeta}$ yields
\begin{equation}\label{eq.ls}
 A_m\bf{\zeta} =\bf{\hat{y}}_m,
\end{equation}
where $A_m = A(\mathbb{X}, \bf{\rho}_m, \bf{\gamma}_m)$ and $\bf{\hat{y}}_m  =  \bf{\hat{y}}(\mathbb{X}, \bf{\rho}_m, \bf{\gamma}_m)$. The values of the unknown coefficients $\bf{\zeta}$ can be revealed solving this linear system.  However, integration with one test function $\phi(\bx, \bf{\rho}_i, \bf{\gamma}_i)$ only reveals part of the information from the given data. Hence, we build the linear systems over the whole test function collection $\mathbb{C}^d_{\bf{\rho}, \bf{\gamma}}$. Namely, we assemble and solve the stacked linear equation 
\begin{equation}\label{eq.large}
\tilde{A} \bf{\zeta} = \bf{\tilde{y}},
\end{equation}
where 
\begin{equation}
\tilde{A} = \begin{pmatrix}
 A_1\\ A_2\\ \vdots\\ A_M   
\end{pmatrix},\; \text{and}\;
\bf{\tilde{y}} = 
\begin{pmatrix}
 \bf{\hat{y}}_1\\
 \bf{\hat{y}}_2\\
 \vdots\\
 \bf{\hat{y}}_M
\end{pmatrix}
\end{equation}
are constructed by the matrix $\{A_m\}_{m=1}^M$ and the vector $\{\hat{\bf{y}}_m\}_{m=1}^M$ over all of the test functions in the collection $\mathbb{C}^d_{\bf{\rho}, \bf{\gamma}}$.

To better approximate the complete test functional space, the collection $ \mathbb{C}^d_{\bf{\rho}, \bf{\gamma}}$ needs to be large enough, which means a huge amount of computational cost. To overcome this difficulty, we borrow the ideas of the functional basis from the Galerkin methods and the collocation points from the collocation methods. It is key to generate the collection of the test functions $\mathbb{C}^d_{\bf{\rho}, \bf{\gamma}}$ taking into account the efficiency and the accuracy. For simplicity, we randomly generate $M$ parameters of $\bf{\rho}_m$ in the range of the values of the data $\mathbb{X}$. Thus the test function collection $\mathbb{C}^d_{\bf{\rho}, \bf{\gamma}}$ is obtained and a large linear system \eqref{eq.large} is then built.

Followed by the linear regression and sparse identification algorithm (STRidge) \cite{kutz16sparse},  the explicit form of the solved dynamics fits the data, see the STRidge Algorithm \ref{alg:STRidge} in appendix for details. Specifically, to give the sparse results, the linear regression is applied but with the hard threshold $\eta$ taken, i.e., smaller components of $\bf{\zeta}$ less than $\eta$ would be set zero and then without those rows and columns, linear regression continues until converges.  Distributing the components of $\bf{\zeta}$ to the drift and diffusion terms gives the explicit form of the hidden stochastic dynamics.

The result of the WCR method error analysis is summarized in Theorem \ref{thm:w}.

\begin{theorem}\label{thm:w}
    Let $\hat{\boldsymbol{\zeta}}=\Abf^\dagger\bbf$ be the learned model coefficients and $\boldsymbol\zeta^\star$ the true model coefficients, $\alpha$ is the order of linear multi-step method. For $C$ independent of sample number $N$ and time interval $\Delta t$, the following holds
    \[\Ebb\left[\nrm{\hat{\boldsymbol{\zeta}}-\boldsymbol\zeta^\star}_\infty\right] \leq C(\frac{1}{\sqrt{N}\Delta t}+\Delta t^\alpha).\]
\end{theorem}

The process of the error analysis is based mainly on the approach in \cite{messenger2022learning}. Following this method, the error of WCR is divided into two parts, each of which is separately estimated. And the complete proof is included in Appendix \ref{app.error}.

Furthermore, for fixed number of basis $b$, the computational complexity of WCR is $\mathcal{O}(LNMd)$, with the number of snapshots $L$, samples $ N $, test functions $ M $ and dimension $ d $. See Appendix \ref{app.complexity} for the detail.
 
The pseudo code of the algorithm is exhibited in Algorithm \ref{alg.WCR}. Figure \ref{fig.scheme} shows the whole procedure of the methodology.

\begin{remark}
Another choice of the ansatz for the drift and diffusion terms can be the neural network leveraging its so-called universal approximation properties. But it would require nonlinear optimization, which enhance the complexity of the algorithm, and left for further investigations beyond the scope of this work. We focus on introducing the framework of the weak form to reveal the hidden dynamics in this work.
\end{remark}
\begin{remark}
    We drew inspiration from the Kernel Density Estimation method and convolutional neural networks when selecting the Gaussian function as our test function. The Gaussian function offers a simple form and can be easily adjusted using only two parameters: the mean and standard deviation. This flexibility allows it to adapt to different data distributions effectively. Moreover, the Gaussian function demonstrates excellent smoothness and performs well in probability density estimation tasks. Its concentrated values and rapid decay make it suitable for the proofs presented in the appendix of our work. However, we would like to emphasize that our method is not limited to Gaussian functions alone and can be applied with other types of test functions as well.
\end{remark}

\begin{figure}[htp]
    \includegraphics[width=\textwidth]{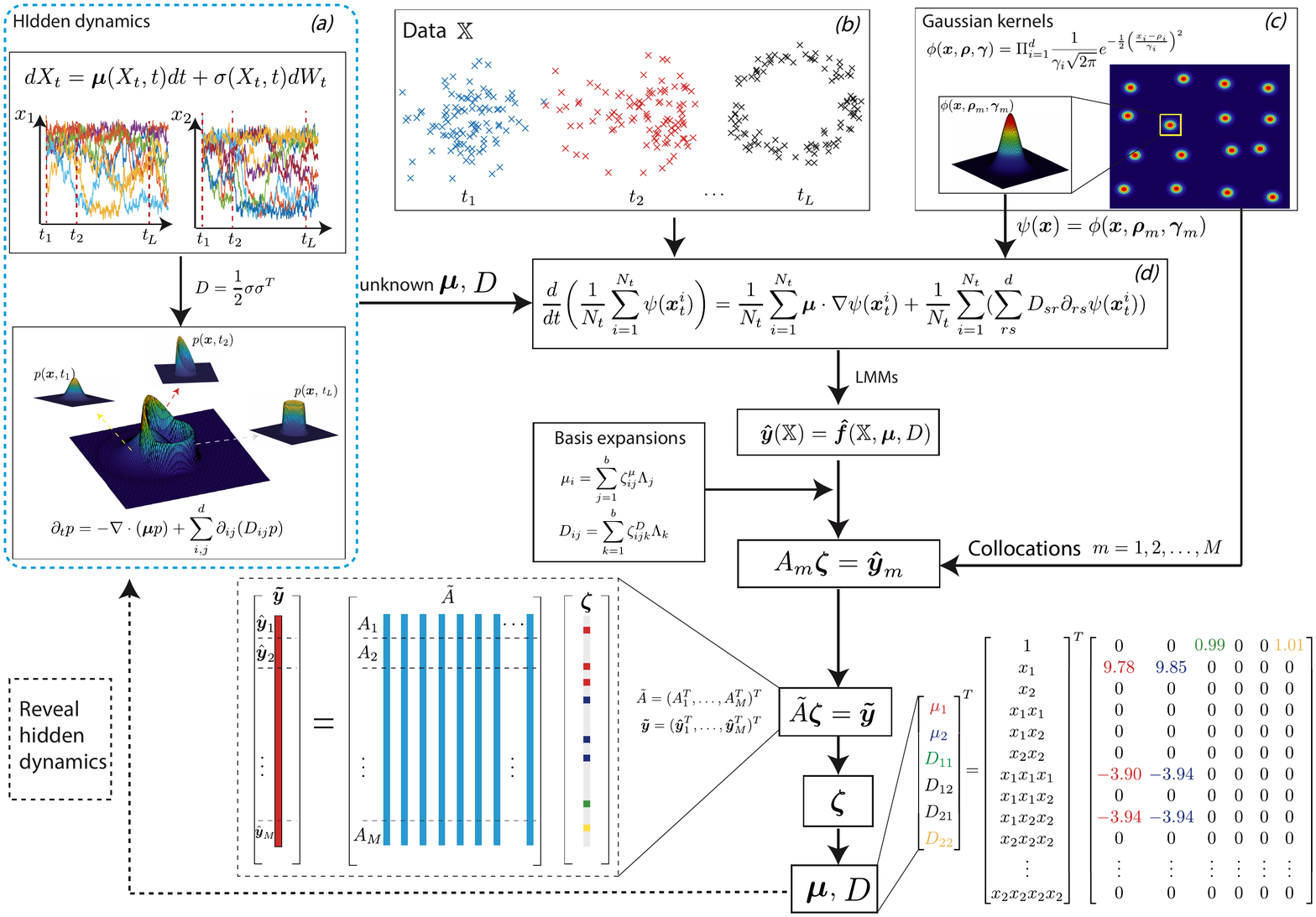}
    \caption{The diagram of the weak collocation regression method. The aggregate data set $\mathbb{X}$ on panel (b) is the collection of $L$ snapshots of samples at time $t_1, t_2, \ldots, t_L$ from one unknown stochastic process. We model this process by the stochastic differential equations in panel (a) with unknown drift $\bf{\mu}(X_t, t)$ and diffusion $\bf{\sigma}(X_t,t)$ terms. By sampling Gaussian kernels in panel (c), for each kernel, the weak form in panel (d) gives the algebraic relation of the unknown terms and the data set. By the LMMs and the  basis expansion of the unknown terms, a linear system is built and further combined together to form a large system over all of the collocation kernels. Finally, the sparse linear regression gives the sparse regression of the drift and diffusion terms and hence the hidden dynamics is revealed.}
    \label{fig.scheme}
\end{figure}

\begin{algorithm}[H] \caption{Weak Collocation Regression method (WCR).} \label{alg.WCR}
\SetAlgoLined
\KwResult{The explicit form of the governing stochastic equation of the data.}
\KwIn{Aggregate data set $\mathbb{X}$.}
Sample $M$ expectations $\{\bf{\rho}_m\}_{m=1}^M$ in the region which contains all of the samples in the data set $\mathbb{X}$;\\
Set the expectations $\{\bf{\gamma}_m\}_{m=1}^M$ as hyper-parameters with default 1 in each dimension;\\
Generate the Gaussian function collections $\mathbb{C}^d_{\bf{\rho}, \bf{\gamma}}=\left\{\phi(\cdot, \bf{\rho}_m, \bf{\gamma}_m) \right\}_{m=1}^M$ using above expectations and deviations;\\


\For{$m <= M$}{
Compute and assemble the vector $\bf{y}(\mathbb{X})=
\begin{pmatrix}
\vdots\\
\frac{1}{N_{t_l}}\sum_{i=1}^{N_{t_l}}\phi(\bx^i_{t_l},\bf{\rho}_m, \bf{\gamma}_m)\\
\vdots\\
\end{pmatrix}
$ by the left term of the weak form \eqref{eq.monte} where the test function is $\psi = \phi(\bx^i_{t_l},\bf{\rho}_m, \bf{\gamma}_m)$ and $l$ is the index of the $l$-th time snapshots;\\

Compute and assemble the matrix $B(\mathbb{X})$ in \eqref{eq.B} by the right term of the weak form \eqref{eq.monte}
over all of the time snapshots with the help of the basis expansions of each entry of drift vector $\bf{\mu}=[\mu_i]$ and diffusion matrix $D=[D_{ij}]$ 
\begin{equation*}
\mu_i = \sum_{j=1}^{b} \zeta_{ij}^\mu \Lambda_j, 
\quad\text{and } D_{ij} = \sum_{k=1}^{b} \zeta_{ijk}^D \Lambda_k
;
\end{equation*}
\\

Assemble the vector $\bf{\zeta}$ by $\bf{\zeta}^\mu = [\zeta_{ij}^\mu]$ and $\bf{\zeta}^D = [\zeta_{ijk}^D]$ using the flatten scheme in \eqref{eq.zeta};\\

Compute the matrix $A_m$ and $\hat{\bf{y}}_m$ by the matrix $B(\mathbb{X})$ and $\bf{y}(\mathbb{X})$ using the LMMs, e.g., the trapezoidal rule \eqref{eq.trapezoidal};\\
}
Stack the linear system $\tilde{A}\bf{\zeta} = \bf{\tilde{y}}$
with $\tilde{A}=(A_1^T,\ldots, A_M^T)^T$ and $ \bf{\tilde{y}} = (\bf{\hat{y}}_1^T, \ldots, \bf{\hat{y}}_M^T)^T$;\\
Compute $\bf{\zeta}$ by a sparse linear regression of $\tilde{A}\bf{\zeta} = \bf{\tilde{y}}$;\\
Reconstruct the drift $\bf{\mu}$ and diffusion $D$ by distributing $\bf{\zeta}$. 
\end{algorithm}

\section{Numerical experiments}
\textbf{Data acquisition. } All the raw data used in this work for the experiments of revealing the hidden dynamics are obtained by integrating the given SDEs
\begin{equation*}
    d {X}_{t}=\boldsymbol{\mu}_{t} d t+ \bf{\sigma}_{t} d W_{t}, \quad t \geq 0
\end{equation*}
from $t=0$ to $t=T$ with Euler–Maruyama scheme
\begin{equation}\label{eq.maruyama}
\tilde{{X}}_{(i+1) \delta t} =\tilde{{X}}_{i \delta t}+\boldsymbol{\mu}_{t} \delta t+\boldsymbol{\sigma}_{t} \sqrt{\delta t} \bf{\mathcal{N}}_{i}.
\end{equation}
Here $\delta t$ is the time step of the numerical scheme, $\bf{\mathcal{N}}_{i}$ are i.i.d standard Gaussian random variables and the intial values are sampled from a given distribution such as a Gaussian function.  The same random seed is adopted among different experiments. The experimental data is then sampled from these trajectories, i.e., $N_i$ points are sampled at each time snapshot to remove the trajectory information, and only $L$ time snapshots are chosen as the experimental data, denoted as $\mathbb{X}=\{\mathbb{X}_i\}_{i=1}^L=\{\{\bf{x}_i^j\}_{j=1}^{N_i}\}_{i=1}^L$. The time snapshots can be non-equally spaced where the time interval $\Delta t_i=t_{i+1} - t_i$ varies and if the time interval is equal we denote the interval as $\Delta t$. The random noise is added to the raw data with the noise level $\delta$ as $\hat{\bf{x}}_i^j = \bf{x}_{i}^j + \delta\mathcal{U}_i^j\bf{x}_{i}^j$ where $\mathcal{U}_i^j$ is a random variable. 
\\
\textbf{Experimental setups. }
In this work, we use the collocations of the Gaussian functions as the test functions for the weak form. We sample these Gaussian kernels by randomly sampling the expectations $\bf{\rho}_m$, $m=1,2,\ldots, M$ using Latin Hypercub Sampling (LHS) method \cite{stein1987large}
in the region of the data, i.e., the hypercube containing all of the data. And the standard deviation $\bf{\gamma}=\gamma I_d$, where $I_d$ is the identity matrix, is chosen as a hyper parameter with default $\gamma=1$ in each sampled Gaussian function. In all cases, our experiments show that all of the non-zero coefficients are correctly identified and all of the zero terms are eliminated
by the sparse regression. Hence in this work,
we define the Maximum Relative Error (MRE) of non-zero terms
\begin{equation*}
    \text{MRE} = \max_{\theta_i\neq 0}\frac{|\hat{\theta}_i-\theta_i|}{|\theta_i|} 
\end{equation*}
as the criterion for the evaluation of the experimental results, where $\theta_i$ represents the $i$-th parameter of the drift and diffusion terms and $\hat{\theta}_i$ is the learned parameter from data.
All the experiment were done on the  MacBook Pro 2021 with an M1 chip. We summarize the notations used throughout the experiments in Table \ref{tab.symbol}.
\begin{table}[htp]
    \centering
    \begin{tabular}{cccc}
    \toprule
    Variable & definition & 
    Variable & definition \\
    \midrule
    $L$ & Number of time snapshots& $\Delta t$ & Time interval of snapshots\\
    $\gamma$ & Gaussian variance & $\bf{\rho}_m$ & Expectation of the $m$-th Gaussian kernel
    \\
    $M$ & Gaussian sample number & $N_i$  & Sample number at $i$-th time snapshot \\
    MRE & Maximum Relative Error  & $\delta$ & Noise level$^*$  \\
    \bottomrule
    \end{tabular}
    \caption{Notations used throughout the experiments. $^*$We add multiplicative noise by $x=x+\delta\mathcal{U}x$, where $\mathcal{U}$ is an uniform random variable in [-1,1]. See section \ref{sec.performance} for details.}
    \label{tab.symbol}
\end{table}

\subsection{Typical 1-dimensional problem} \label{sec.1d}

One-dimensional stochastic problem widely exists in the scientific and engineering fields, such as the population growth, asset price and investments \cite{oksendal2013stochastic}. To better illustrate the abilities of our WCR method for dealing with complex tasks, in this section, we focus on revealing the hidden dynamics from 1d aggregate data avoiding the difficulties brought by the dimension of the data. We consider the following 1d model
\begin{equation}
dX_t = \mu(X_t)dt + \sigma(X_t)dW_t,
\end{equation}
where drift $\mu(X_t)$ and diffusion  $\sigma(X_t)$ terms reduce to scalar functions. Here we mainly focus on the five cases: (I) Cubic polynomial problem with only three snapshots; (II) Variable-dependent diffusion problem; (III) Quintic polynomial problem with high contrast; (IV) General form of the basis dictionary; (V) General drift term out of basis.
And these cases give a direct illustration of the good performance of WCR on the complex tasks. 

\textbf{(I) Cubic polynomial problem with three snapshots.} In this case, the raw data of the experiment are generated by the 1d cubic polynomial form
\begin{equation} \label{eq.x-x3}
    dX_t = (X_t-X_t^3)dt + dW_t
\end{equation}
in from $t=0$ to $t=1$ with initial values at $t=0$ sampled from a Gaussian distribution $ \mathcal{N}(0,0.1)$
using Euler-Maruyama scheme \eqref{eq.maruyama}. We obtained $10, 000$ samples at each snapshot, and only three snapshots at (a) $t=0.1, 0.3, 0.5$ and (b) $t=0.2, 0.5, 1$ are adopted as the experimental data. In this case, we expand the drift term as the third-order basis expansion form 
$$ \mu(x) = \lambda_0 + \lambda_1x + \lambda_2x^2 + \lambda_3x^3, $$
and the diffusion term is treated as a tunable parameter $ D(x) = \frac{1}{2}\sigma^2 = D_0$. 
We have used the same experimental setup for the same problem of Chen's work \cite{chen2021solving}.

 Only three snapshots are available in the data set (a) and (b). Further the time intervals of the snapshots in (b) are not equal. Thus the approximation of the temporal derivatives requires a reliable scheme to overcome the difficulties. Here we use the variable step-size version of Implicit Adams methods of trapezoidal rule \eqref{eq.nonequalLMM2}, which can be applied on small amount of time snapshots and non-equally spaced time series data. For the collocation kernels, 20 Gaussian kernels are sampled by sampling the expectations of the Gaussian function using LHS method with the standard variance set as $\gamma=0.85$.
 
 The results are shown in Figure \ref{fig.1d3.PKduan} and Table \ref{tab.1d3.PKduan}. We have compared the results with Chen's results \cite{chen2021solving} as the state of the art (sota) to show that for the same problem, WCR method can achieve a comparable accuracy but with a much less computational cost. All of the non-zero terms are correctly identified and all of the zero terms are eliminated by the sparse regression. The MRE of (a) and (b) are less than $4.6\%$ and $0.42\%$ respectively. Most notably, the experiments of applying WCR on the data set (a) and (b) are all completed within 0.02s on a MacBook Pro 2021 with an M1 chip. 

For a mild amount of the data, WCR method still achieves remarkable performances. 
We sampled 1000 points each snapshot in the case (b) above without changing other setups to test our framework. In this case, WCR method still achieves a good result with the MRE less than $4\%$ within only 0.006s on the MacBook pro.

\begin{figure}[htp]
\centering
\subfigure[]{\includegraphics[width=0.45\textwidth]{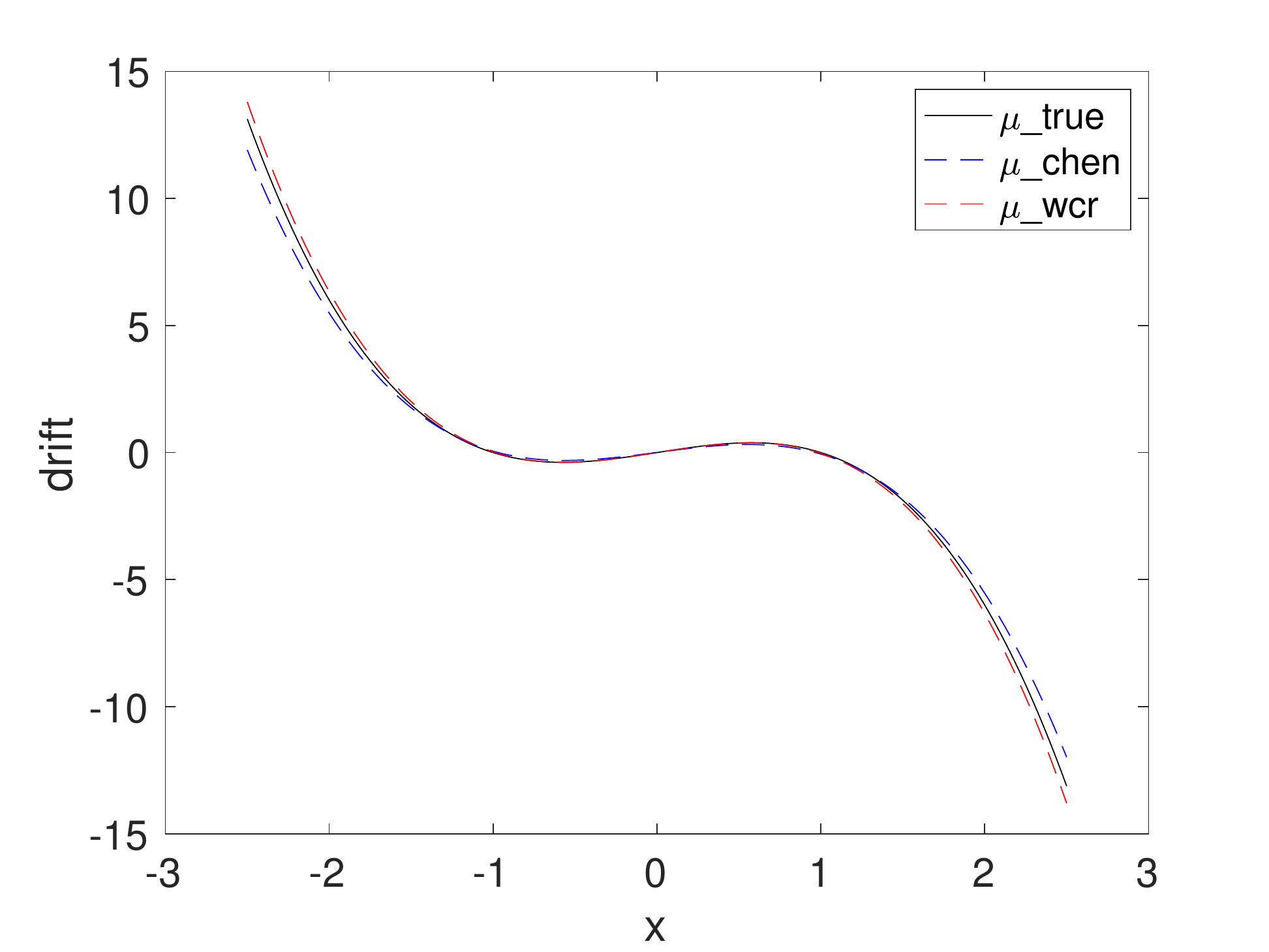}}
\subfigure[]{\includegraphics[width=0.45\textwidth]{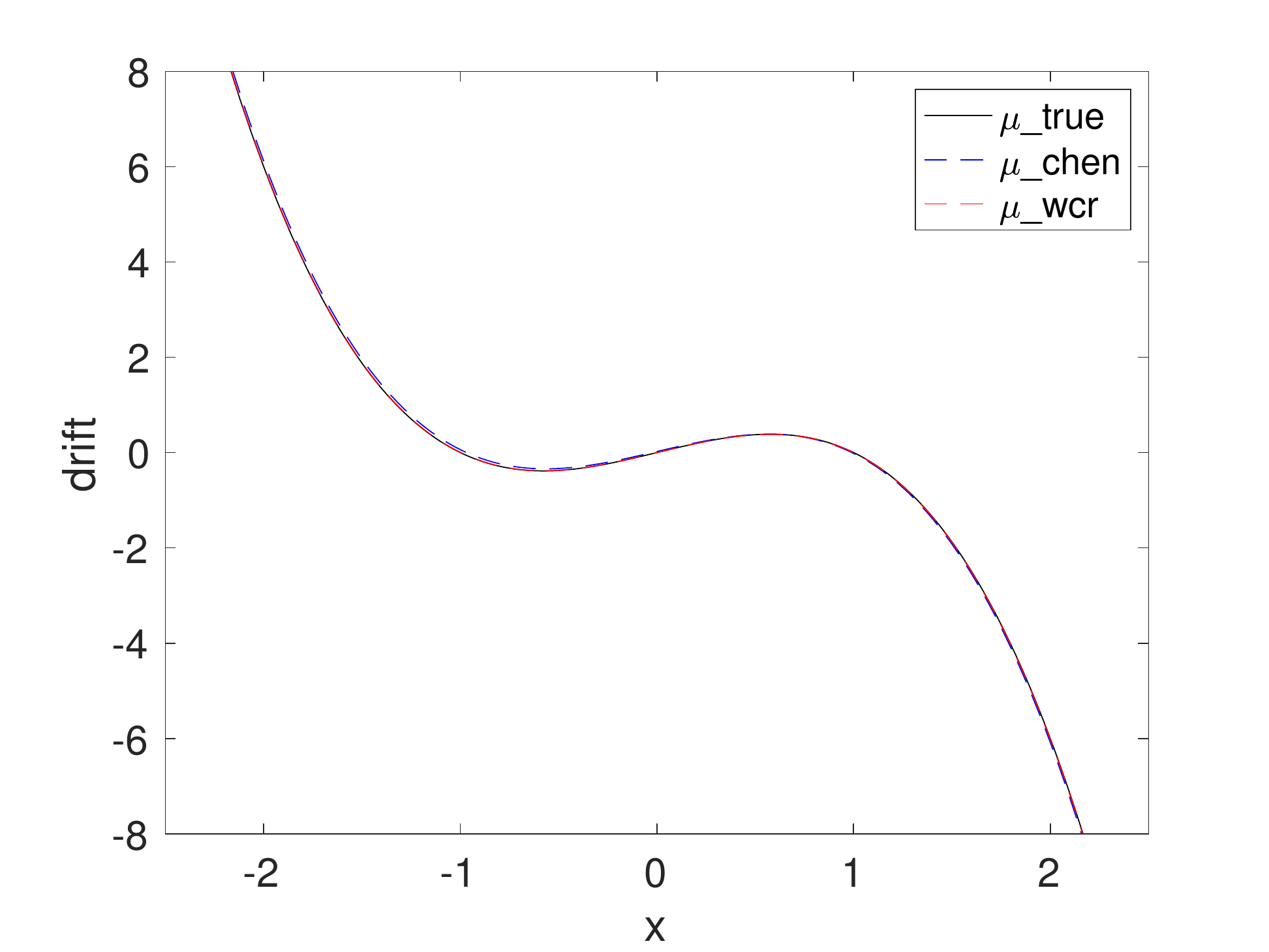}}
\caption{The results of 1d cubic polynomial problem compared with Chen's work (sota). Reveal the unknown drift and diffusion terms with 10000 samples of $X_t$ at different time snapshots: (a) Observations at $t = 0.1, 0.3, 0.5$; (b) Observations at $t=0.2, 0.5, 1$ where the samples are generated by the given SDE with drift term $\mu\_\text{true} = x-x^3$ and diffusion term $\sigma=1$. The inference results are denoted by $\mu\_{\text{true}}$, $\mu\_{\text{chen}}$ (sota) and $\mu\_{\text{wcr}}$ (ours).} \label{fig.1d3.PKduan}
\end{figure}

\begin{table}[htp]
\centering 
\begin{tabular}{cccccc}
\toprule Parameter & $\lambda_{0}$ & $\lambda_{1}$ & $\lambda_{2}$ & $\lambda_{3}$ & $\sigma$ \\
\hline
True parameters & 0 & 1 & 0 & $-1$ & 1 \\
\hline
(a) Chen & $0.0051$ & $0.8422$ & $-0.0071$ & $-0.8994$ & $1.0347$ \\
(a) WCR & $0$ & $\bf{1.0160}$ & $0$ & $\bf{-1.0457}$ & $\bf{1.0127}$ \\
\hline
(b) Chen & $0.0225$ & $0.9638$ & $-0.0010$ & $-1.0035$ & $1.0138$ \\
(b) WCR & $0$ & $\bf{0.9982}$ & $0$ & $\bf{- 1.0019}$ & $\bf{1.0042}$ \\
\bottomrule
\end{tabular}
\caption{The results of 1d cubic polynomial problem compared with Chen's work (sota). Reveal the unknown drift and diffusion terms with 10000 samples of $X_t$ at different time snapshots: (a) Observations at $t = 0.1, 0.3, 0.5$; (b) Observations at $t=0.2, 0.5, 1$ where the samples are generated by the given SDE with drift term $\mu\_\text{true} = x-x^3$ and diffusion term $\sigma=1$. The inference results are denoted by Chen (sota) and WCR (ours).} \label{tab.1d3.PKduan}
\label{tab.1d}
\end{table}

For sparse identification, its concept is rooted in the Occam's Razor principle\cite{domingos1999role, walsh1979occam}, which suggests that we should prioritize selecting the simplest model when uncovering the underlying dynamics of data. By choosing a model with the fewest non-zero coefficients, we not only enhance interpretability but also mitigate the risk of overfitting the data, thereby improving the generalization capability of the model. In our framework, sparsity serves as an optional feature rather than an essential requirement. If sparsity is not necessary for a particular analysis, a linear regression can still effectively reveal hidden dynamics. To demonstrate this, we conducted an experiment in case (b) and we replaced the sparse identification technique with a linear regression, achieving satisfactory results in Table \ref{tab.1d_sparse}.

\begin{table}[h]
    \centering
    \begin{tabular}{cccccc}
        \toprule
        Parameter & $\lambda_0$ & $\lambda_1$ & $\lambda_2$ & $\lambda_3$ & $\sigma$ \\
        \midrule
        True & 0 & 1 & 0 & -1 & 1 \\ 
        With sparse identification & 0 & 0.9982 & 0 & -1.0019 & 1.0042 \\ 
        Without sparse identification & 0.0217 & 0.9584 & -0.0118 & -0.9751 & 1.0082 \\
        \bottomrule
    \end{tabular}
    \caption{Experiment in section 3.1(I)(b) with drift term $\lambda_0+\lambda_1x+\lambda_2x^2+\lambda_3x^3$ and diffusion term $\sigma$. The experiment setup is same with section 3.1(I)(b).}
    \label{tab.1d_sparse}
\end{table}

\textbf{(II) Variable-dependent diffusion term problem.}  
Many attentions in the existing literature are paid to the problem with constant diffusion in revealing the hidden dynamics from data. But a more general setting is the non-constant diffusion, i.e., the variable-dependent diffusion. In this part, we would show that WCR method also works for the variable-dependent case. The variable-dependent stochastic equation 
\begin{equation*}
    dX_t = (X_t-X_t^3)dt + (1+X_t)dW_t
\end{equation*}
is considered as the true model to generate the raw data of the snapshots at tme $t=0, 0.2, 0.5, 1$.
Our task is to reveal the true model from the data by the parameterized SDE form 
\begin{equation}\label{eq.gs}
    dX_t = (\lambda_0 + \lambda_1 X_t + \lambda_2 X_t^2 + \lambda_3 X_t^3)dt + (\sigma_0+\sigma_1X_t)dW_t,
\end{equation}
where $\lambda_i, i=0,1,2,3$ and $\sigma_j, j=0,1$ are the tunable parameters. Note that, the weak form of the Fokker-Planck is leveraged in our method, where the drift term $\mu(x)$ and the $1\times 1$ diffusion matrix $D$ is revealed. By setting $D(x)=b_0 + b_1 x + b_2x^2$, WCR method gives the approximation of the values of $\lambda_i, i=0,1,2,3$ and $b_j, j=0,1,2$. And by the relation $D=\frac{1}{2}\sigma^2$, a nonlinear regression is applied on the diffusion matrix $D$ and the parameters $\sigma_i, i=0,1$ are approximated, giving the explicit form of the gorverning equation \eqref{eq.gs}. The results of the parameters are listed in Table  \ref{tab.1d3.diffusion} and WCR method still works well in the variable-dependent problem.

\begin{table}[htp]
    \centering
    \begin{tabular}{ccccccc}
    \toprule
    Parameter & $\lambda_0$ & $\lambda_1$ & $\lambda_2$ & $\lambda_3$ & $\sigma_0$ & $\sigma_1$ \\
    \midrule
    True & 0 & 1 & 0 & -1 & 1 & 1 \\
    \midrule
    WCR & 0 & 0.9688 & 0 & -1.0264 & 0.9955 & 1.0326 \\
    \bottomrule
    \end{tabular}
    \caption{The results of 1d cubic polynomial problem when diffusion term is not a constant. Reveal the unknown drift and diffusion terms with 10000 samples of $X_t$ at $t=0,0.2,0.5,1$ where the samples are generated by the given SDE with drift term $\mu\_\text{true} = x-x^3$ and diffusion term $\sigma=1+x$.} \label{tab.1d3.diffusion}
\end{table}

\textbf{(III) Quintic polynomial drift problem.}
Quintic polynomial drfit has fifth-order polynomial terms, raising the difficultis of the revealing the true form of the hidden stochastic dynamics. The problem becomes more subtle when the cofficients of the polynomial terms are in different scale. In this experiment,  we would show that for the high order polynomial drift term with high contrast coefficient, WCR method can still reach a good accuracy of modeling the data. 

The raw data are generated by integrating the following SDE from $t=0$ to $t=10$
\begin{equation}
    dX_t = \mu dt + dW_t,
\end{equation}
where the drift term is the high order polynomial form with high contrast coefficient as
\begin{equation}
   \mu \overset{\Delta}{=} -x(x-1)(x-2)(x-3)(x-4) = -24x+50x^2-35x^3+10x^4-x^5.
\end{equation}

In this experiment, WCR is applied on the aggregate data $\mathbb{X}$ to reveal the hidden stochastic dynamics. Two kinds of the data with time interval (a) $\Delta t =0.1$ and (b) $\Delta t= 0.5$ with $N=5000$ points in each time snapshot are considered. To model the data, we take the fifth order of the polynomial expansion for the drift term with 
\begin{equation*}
\mu(x)=\theta_0+\theta_1x+\theta_2x^2+\theta_3x^3+\theta_4x^4+\theta_5x^5,
\end{equation*}
where $\theta_i, i=0,1,\ldots 5$ are the tunable parameters to be determined by the data. The diffusion term is chosen as one tunable parameter $D_0$. 

In the collocation of the kernels, 200 Gaussian functions are sampled by randomly generating their expectations $\{\rho_i\}_{i=1}^{200}$ in the region of the data with the default standard variance $\gamma=1$. For the approximation of the temporal derivatives, the Milne method is applied. The computation costs within 2 seconds and the results are summarized in Table \ref{tab.1d5}. From the result, we can see that WCR method achieves a good performance in the higher problem with less than 3\% Max Relative Error (MRE) within seconds on the Macbook Pro.  
\begin{table}[htp]
    \centering
    \begin{tabular}{cccccccccc}
        \toprule
        Settings & \multicolumn{6}{c}{drift coefficients $\{\theta_i\}_{i=0}^5$} & diffusion $D_0$ & MRE & Time (s) \\
        \midrule
        True & 0 & -24 & 50 & -35 & 10 & -1 & 1 & - & -\\
        \midrule
        (a) $\Delta t=0.1$ & 0 & -23.8 & 49.6 & -34.7 & 9.92 & -0.991 & 1.002 & 0.87\% & 1.8 \\
        (b) $\Delta t=0.5$ & 0 & -24.9 & 51.6& -35.9 & 10.2 & -1.018 & 1.029 & 3.82\% & 0.4 \\
        \bottomrule
    \end{tabular}
    \caption{The results of one-dimensional high contrast problem with different time snapshot. Reveal the unknown drift and diffusion terms with (a) $\Delta t=0.1$; (b) $\Delta t=0.5$, where the samples are generated by the true SDE with drift term $\mu(x)=-24x+50x^2-35x^3+10x^4-x^5$ and diffusion term $\sigma=1$, gaussian functions $M=200$, samples number $N=5000$.} 
    \label{tab.1d5}
\end{table}

\textbf{(IV) General form of the basis dictionary.}
As mentioned in Section 2.4, the dictionary represents are adopted for the unknown terms. In the above experiments, we simply choose the polynomial basis as the dictionary for their simplicity and interpretability. And polynomials were primarily employed as the basis in the experimental investigations conducted in this paper. Nevertheless, it should be noted that the proposed method is not restricted to polynomial bases and can be extended to more general functions. In this part, we depict a more general function dictionary which consists of both polynomial and trigonometric functions to demonstrate the abilities of our methods for the general form of the approximations.

The experimental data in this subsection was obtained by discretizing the following stochastic differential equation \eqref{eq.cos} using the Euler-Maruyama method with a step size of 0.1 from $t=0$ to $t=1$. 
\begin{equation} \label{eq.cos}
    dX_t = (X_t+\cos(3X_t))dt+dB_t
\end{equation}

Then we expand the drift term using the following composite basis $ \Lambda $.
\begin{equation}
\begin{gathered}
    \Lambda = \left\{
    \begin{pmatrix} 1 \\ x \\ x^2 \\ x^3 \end{pmatrix} \otimes
    \begin{pmatrix} 1 \\ \cos{x} \\ \cos{2x} \\ \cos{3x} \end{pmatrix}
    \right\} = \left\{
    \begin{pmatrix}
        1 & \cos{x} & \cos{2x} & \cos{3x} \\
        x & x\cos{x} & x\cos{2x} & x\cos{3x} \\
        x^2 & x^2\cos{x} & x^2\cos{2x} & x^2\cos{3x} \\
        x^3 & x^3\cos{x} & x^3\cos{2x} & x^3\cos{3x}
    \end{pmatrix} \right\} \\ 
    \xlongequal{\mbox{flatten}}
    \{1, \cos{x}, \cos{2x}, \cos{3x}, x, x\cos{x}, x\cos{2x}, x\cos{3x}, \cdots, x^3\cos{2x}, x^3\cos{3x}\}
\end{gathered}
\end{equation}
The basis is composed of a product of polynomials with degree no more than 3 and cosine functions with frequencies of 0, 1, 2, and 3. Therefore, there are a total of 16 terms, and with the addition of diffusion represented by a constant, there are 17 parameters to be solved. For the selection of test functions, 20 Gaussian functions with a variance of 1 were chosen, and their means were sampled from the data region using the Latin Hypercub Sampling method. There are 10,000 samples at each time point. The results are shown in Table \ref{tab.cos}. It can be seen that all the redundant base terms have been successfully eliminated by sparse regression. In the face of the problem of trigonometric basis, the WCR method can still achieve good performance with less than 4\% maximum relative error.

\begin{table}[h]
    \centering
    \begin{tabular}{cccccc}
        \toprule
        coefficient & $x$ & $\cos{3x}$ & other drift terms & diffusion & MRE \\ 
        \midrule 
        True & 1 & 1 & 0 & 1 & - \\ 
        Learned & 1.029 & 1.038 & 0 & 1.0043 & 3.8\% \\ 
        \bottomrule
    \end{tabular}
    \caption{The results of one-dimensional trigonometric basis drift problem. Reveal the unknown dynamics with samples $N=10000$ and gaussian functions $m=20$. The drift term is $\mu(x)=x+\cos{3x}$ and diffusion term $\sigma=1$.}
    \label{tab.cos}
\end{table}

\textbf{(V) General drift term out of basis.}
To illustrate the performance of the basis expansion, we investigate the scenario of the complex drift term without complete expansion of the basis. Namely, in this subsection, we try to reveal the hidden dynmaics from data of the given form 
\begin{equation}
    dX_t = -2X_te^{-X_t^2}dt + dW_t
\end{equation}
where the drift term $-2xe^{-x^2}$ can not be completely expanded in finite polynomial basis. However, we would show that with the order of the polynomial basis increases, the revealing results still meet the needs.

We compare the $L_2$ relative error between the approximate solution under polynomial basis and the true drift terms. The data is observed at $t=0,0.1,0.2,\cdots,1$ with 10,000 samples at each snapshot.

The diffusion terms and $L_2$ relative errors of the drift terms under different orders of polynomial basis are presented in Table \ref{tab.1d_order}. Figure \ref{fig.1d_order} illustrates the functional graph of the drift terms. The calculation interval for relative error is $[-1, 1]$, because our method is a supervised learning approach, and the values of drift terms in regions without data are inherently unlearnable. As the order of polynomial basis increases, the results obtained by WCR gradually approach the true values of the drift terms. The best learning performance is achieved when the order is 9, with a relative error of 0.05\% and a diffusion term of 0.9994.

\begin{table}[h]
    \centering
    \begin{tabular}{cccccccc}
        \toprule
        basis order & 3 & 4 & 5 & 6 & 7 & 8 & 9  \\
        \midrule
        $L_2$ relative error & 0.0734 & 0.0728 & 0.0242 & 0.0205 & 0.0042 & 0.0008 & 0.0005 \\
        diffusion & 0.9331 & 0.9332 & 0.9653 & 0.9699 & 0.9863 & 0.9984 & 0.9994 \\ 
        \bottomrule
    \end{tabular}
    \caption{The results of the drift terms $-2xe^{-x^2}$ under various orders of polynomial basis. The $L_2$ relative error of the drift terms is calculated over the interval $[-1, 1]$. The data was obtained from 10,000 observations at $t=0,0.1,0.2,\cdots,1$. And 20 Gaussian functions were used for revealing the dynamics.}
    \label{tab.1d_order}
\end{table}

\begin{figure}
\centering
\includegraphics[width=0.55\textwidth]{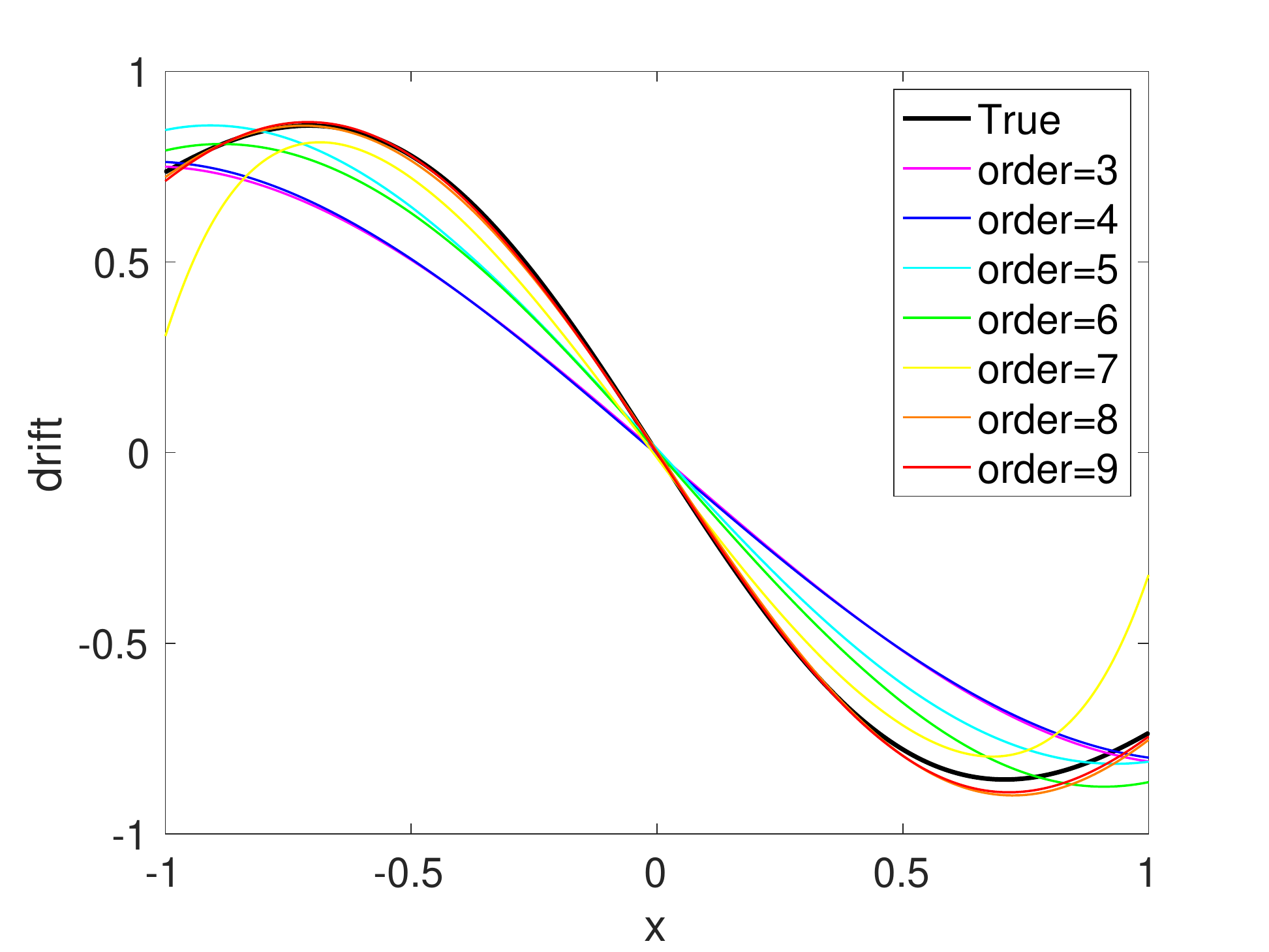}
\caption{Functional graphs of the drift terms $-2xe^{-x^2}$ revealing by WCR method under different orders of polynomial basis from 3 to 9.}
\label{fig.1d_order}
\end{figure}


\subsection{2-dimensional problem with coupled drift terms}
Sombrero potential $V$ is a well-known potential form that the symmetry breaking is triggered in the quantum mechanics. The gradient of the potential acts as the force and gives the drift term of the stochastic process as $\mu = \nabla V$.
 The drift term induced by the Sombrero potential has the coupled terms in each dimension, i.e., $\mu_1$ and $\mu_2$ are not independent. To illustrate the abilities of revealing hidden stochastic dynamics with coupled drift terms, we apply WCR method on the aggregate data generated by the 2-dimensional Brownian motion with the coupled drift term 
 \begin{equation}\label{eq.coupled}
  \bf{\mu} = - \nabla V = 
  \begin{pmatrix}
  10x_1-4x_1^3-4x_1x_2^2\\
  10x_2-4x_2x_1^2-4x_2^3
\end{pmatrix},
 \end{equation}
 which is induced by the gradients of the Sombrero potential  
\begin{equation}
    V = -5\|\bf{x}\|^2 +  \|\bf{x}\|^4= -5(x_1^2 + x_2^2) + (x_1^2 + x_2^2)^2.
\end{equation}


The true model decorated by the coupled drift \eqref{eq.coupled} has the SDE form
 \begin{equation}\label{eq.som}
    d\bf{x} =  
    \begin{pmatrix}
        c_{11}x_1^1 + c_{12}x_1^3 + c_{13} x_1x_2^2 \\ 
        c_{21}x_2^2 + c_{22} x_1^2x_2 +c_{23} x_2^3
    \end{pmatrix} dt +\begin{pmatrix} \sigma_1& 0\\ 0 & \sigma_2\end{pmatrix} dW_t,
\end{equation}
where $c_{11}=c_{21}=10$, $c_{12}=c_{13}=c_{22}=c_{23}=-4$ corresponding to the coupled drift term \eqref{eq.coupled} and $\sigma_1=\sigma_2=1$ for simplicity. 

The raw data are generated by integrating the SDE from $40, 000$ initial samplings from the Gaussian distributions at $t=0$ to $t=15$ using Eular-Maruyama method. Take $N=20, 000$ samples out of the total $40, 000$ points in each snapshot and collect totally $L=151$ snapshots at $t=0, 0.1, \ldots, 15$, we have the aggregate data set $\mathbb{X}=\{\mathbb{X}\}_{i=1}^L=\{\{\bf{x}_i^j\}_{j=1}^{N}\}_{i=1}^L$ for the experiment. In this subsection, we apply WCR method on the data to reveal the true model and give a direct illustration of the performance of our method.

To reveal the hidden dynamics of the 2-dimensional stochastic process, $200$ gaussian kernels are randomly sampled in the regime of the data and the Milne method \eqref{eq.am} is used to approximate the temporal derivatives.  The elements of the unknown drift and coefficient term are approximated by the linear combinations of the forth-order complete polynomial basis
\begin{equation}
\Lambda = \{1, x_1, x_2, x_1^2, x_1x_2, x_2^2, x_1^3, x_1^2x_2, x_1x_2^2, x_2^3,, x_1^4, x_1^3x_2, x_1^2x_2^2, x_1x_2^3, x_2^4\},
\end{equation}
with $|\Lambda|=15$.  For simplicity we denote the basis vector $\bf{\Lambda}=(1, x_1, x_2, \ldots, x_2^4)^T$ which is the vectorization of $\Lambda$. Note that, the highest order of the drift term is three in the true model, but we have used a higher order term to approximate the unknown terms of both drift and diffusion terms. It follows that the unknown drift and diffusion terms are approximated by the following form
\begin{equation}
    \bf{\mu} = \begin{pmatrix}
    \mu_1\\
    \mu_2\\
    \end{pmatrix} =\begin{pmatrix}\bf{\zeta}_1^\mu\cdot \bf{\Lambda}\\ \bf{\zeta}_2^\mu\cdot \bf{\Lambda} \end{pmatrix}, \quad
    D = \begin{pmatrix} D_{11} & D_{12} \\ D_{21} & D_{22}\end{pmatrix}
    = \begin{pmatrix} \bf{\zeta}_{11}^D\cdot\bf{\Lambda} & \bf{\zeta}_{12}^D\cdot\bf{\Lambda} \\ \bf{\zeta}_{21}^D\cdot\bf{\Lambda} & \bf{\zeta}_{22}^D\cdot\bf{\Lambda}\end{pmatrix},
\end{equation}
where $\bf{\zeta}^{\mu}_{i}\in\mathbb{R}^d,\bf{\zeta}^D_{ij}\in\mathbb{R}^d$ are coefficient vectors with dimension $d=|\Lambda|=15$. Thus for each drift term and each element of the diffusion matrix, we have 15 coefficients, making a total of 90 parameters to be revealed. By a sparse linear regression, instead of the 90 coefficients for a forth-order complete polynomial basis, only 8 coefficients corresponding to those in \eqref{eq.som} are listed for simplicity, as the others are identified correctly as zero. All non-zero terms of the true model are not missed and no redundant coefficients which should be zeros are not superfluous by WCR method. The results are shown in Table \ref{tab.sombrero}. The Max Relative Error (MRE) is about $2.42\%$ of the coefficients. The experiment is done on the Macbook pro within $21$ seconds. The results in Table \ref{tab.sombrero} illustrate the abilities of WCR method for revealing hidden dynamics with coupled terms under high-order expansion. 
\begin{table}
\centering
\begin{tabular}{ccccccccc}
\toprule Parameter & $c_{11}$ & $c_{12}$ & $c_{13}$ & $c_{21}$ & $c_{22}$ & $c_{23}$ & $\sigma_1$ & $\sigma_2$\\
\midrule
True parameters & 10 & -4 & -4 & 10 & -4 & -4 & $1$ & 1 \\
\midrule
 WCR & 9.7823 & -3.9032 & -3.9364 & 9.8495 & -3.9441 & -3.9409 & $0.9881$ & 1.0076 \\
\bottomrule
\end{tabular}
\caption{The results of the 2-dimensional problem with Sombrebro potiential. Reveal the unkown drift and diffusion terms with the expansion of forth-order complete polynomial basis. No redundant coefficients have been learned with only 8 out of the 90 cofficients nonzero corresponding to the true model. The Max Relative Error of the non-zero coefficients is 2.42\%.
}
\label{tab.sombrero}
\end{table}


\subsection{Multi-dimensional problem} \label{sec.multi}
Many methods depending on the integration over space always get stuck when the dimension of the data increases because of the exponential increase of the computational cost. Even under three or four dimension, the integration is computationally expensive, despite of the incapacity for higher dimensions such as 10d or 20d.   In the contrast, WCR method takes a more subtle strategy to avoid the direct computation of the integral. Thanks to the weak form, the spatial derivatives of the probability function have been transferred to the test function, making the integral simply computed by the summation over the samples avoiding the curse of the dimensionality. Hence WCR method can be naturally extended to the higher dimension other than one or two dimension. 

In this subsection, we take the same setup consistent with the 3d and 4d problems in  \cite{chen2021solving} and to show that WCR method can achieve a high accuracy within seconds. The true model is d-dimensional extension of the SDE \eqref{eq.x-x3}  with the form
\begin{equation} \label{eq.d}
    dX^i_t = \left(X^i_t-(X^i_t)^3\right)dt + dW^i_t, \quad i=1,2,\cdots,d,
\end{equation}
where the drift term is $\mu_i = x_i - x_i^3$ in the $i$-th dimension. 
The raw data are generated by the true model using Euler-Mayaruma method.  To get the aggregate data set, $100, 000$ samples each snapshot are taken and the snapshots at (i) $t=0.1,0.3,0.5,0.7,1$ (ii) $t=0.1,0.2,0.3,0.5,0.7,0.9,1$ are collected as the experimental data. 

In the experiment, the unknown drift term in each dimension is approximated by 
\begin{equation}
\hat{\mu}_i = \theta_0 + \theta_1 x_i + \theta_2 x_i^2 + \theta_3 x_i^3, 
\end{equation}
and the diffusion term approximated by a tunable parameter $D_i$,  $i=1,2,\ldots, d$. $100$ Gaussian kernels are sampled by the expectations using LHS method with the default standard deviation where $\gamma=1$.  We list the Max Relative errors and the computational time in Table \ref{tab.3d4d} and the results show our WCR method achieves a high accuracy in the 3d and 4d cases within seconds on the MacBook Pro. Figure \ref(fig.PK3d4d) depicts the learned drift terms compared with the true ones.

\begin{table}[htp]
    \centering
    \begin{tabular}{cccccc}
         \toprule
         & 3D-(i) & 3D-(ii) & 4D-(i) & 4D-(ii)\\
         \midrule
         MRE & 3.39\% & 1.86\% & 7.04\% & 3.18\% \\
         Time(s) & 3.6 & 4.7 & 6.1 & 7.9 \\
         \bottomrule
    \end{tabular}
    \caption{The Max Relative Error of the learned coefficients and the compuational time on the MacBook Pro of revealing the unknown 3d, 4d cubic polynomial problems. The aggregate data $\mathbb{X}$ is composed of the time snapshots at (i) $t=0.1,0.3,0.5,0.7,0.9$ (ii) $t=0.1,0.2,0.3,0.5,0.7,0.9,1$ with $100, 000$ points in each snapshot. The unknown dynamics is approximated by the drift $\hat{\mu}_i = \theta_0 + \theta_1 x_i + \theta_2 x_i^2 + \theta_3 x_i^3$ and diffusion $D_i$ in each dimension with tunable parameters $\theta_i$ and $D_i$, $i=1,2,\ldots, d$.
And  $100$ gaussian kernels are sampled to give the composed linear system.} \label{tab.3d4d}
\end{table}

\begin{figure}
\centering
\subfigure[]{\includegraphics[width=0.48\textwidth]{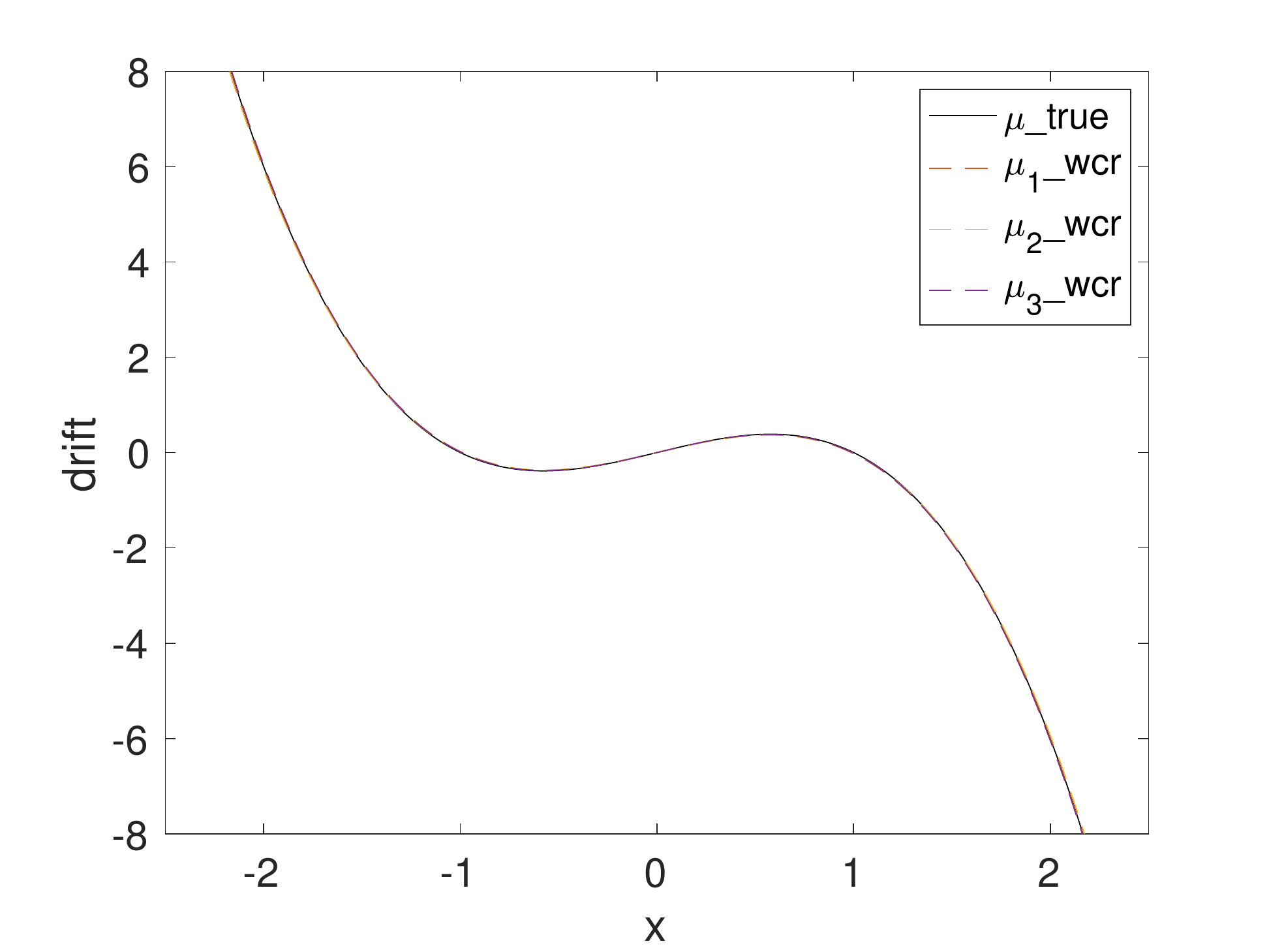}}
\subfigure[]{\includegraphics[width=0.48\textwidth]{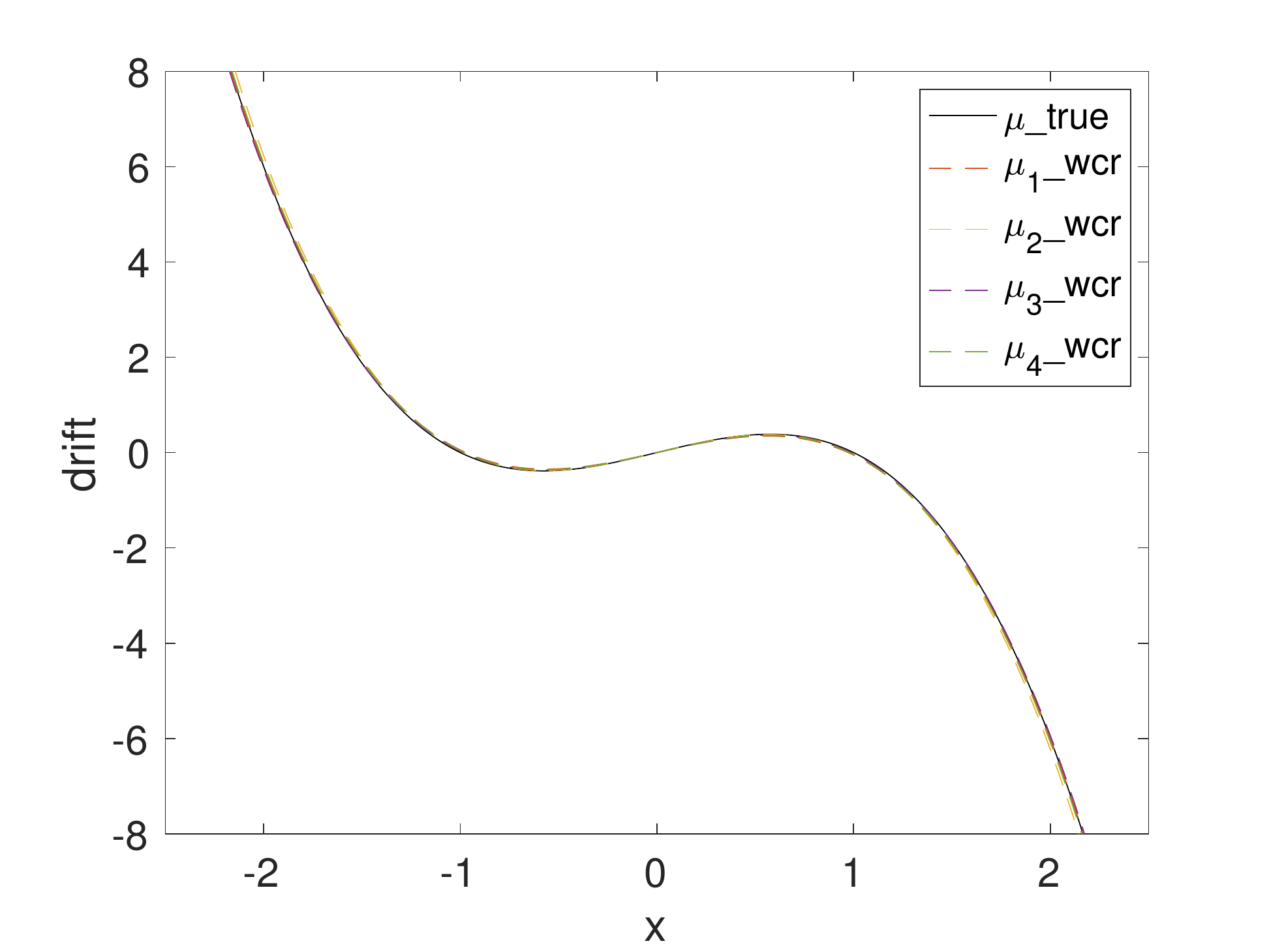}}
\caption{The learned drift terms of the 3d and 4d problems. Reveal the unkown drift and diffusion terms with $100, 000$ samples each snapshot at time $t = 0.1, 0.3, 0.5, 0.7, 1$ of (a) 3-dimensional problem; (b) 4-dimensional problem, where the samples are generated by the given SDE with drift term $\mu\_\text{true} = x-x^3$ and diffusion term $\sigma=1$ in per dimension. The inference results are denoted by $\mu\_\text{true}$ and $\mu_i\_\text{wcr}$ for the $i$-th dimension.} \label{fig.PK3d4d}
\end{figure}

\subsection{High-dimensional problem}

For three or four dimensional data, WCR method shows a high accuracy of revealing the hidden stochastic dynamics. Higher dimensional problem especially over ten dimensional brings more challenges to the modeling task because of the so-called curse of the dimensionality. In this subsection, we would show that WCR method can be applied on the high dimensional problems thanks to the Monte-Carlo approximation of the weak form. We consider the 10 and 20 dimensional aggregate data generated by the true model 
\begin{equation}
    dX_t^i = \left(X_t^i - (X_t^i)^3\right)dt + dW_t^i, 
    \qquad i=1, 2, \cdots, d,
\end{equation}
with $d=10, 20$ using Euler-Maruyama method. Several time snapshots of the data at time $t=0,0.1,0.2,0.3,\cdots,1$ are used as the experiment data. 

In the experiemnt, the fouth order basis is used to approximate the drift term in each dimension as $\mu_i=\theta_0+\theta_1x_i+\theta_2x_2^2+\theta_3x_3^3$ and the diffusion term is approxiamted by a tunable parameter $\sigma_i$, making totally $5d$ tunable parameters to be learned. Milne method is used for the temporal derviatives. Different number of the data samples in each snapshot and the Gaussian kernels are considered in the high dimensional problem.

In the 10-dimensional problem, we consider the four cases for samples and kernels (a) $10, 000$ samples of $X_t$ using $1000$ Gaussian kernels; (b) $100, 000$ samples of $X_t$ using $1000$ Gaussian kernels; (c) $10, 000 $samples of $X_t$ using $10, 000$ Gaussian kernels; (d) $100, 000$ samples of $X_t$ using $10, 000$ Gaussian kernels. The four cases are the combinations of the $10, 000$ and $100, 000$ samples with $1000$ and $10, 000$ Gaussian kernels.  Figure \ref{fig.10d} gives the display of the learned drift terms compared with the true model in each case. From the Figure, it is easy to check that the learned results get better when the number of the sample points and Gaussian kernels increases.
And the Max Relative Error for the best result in case (d)  achieves less than $7.8\%$ within 20 minutes. 
We also employ a coupled system in ten dimensions to showcase the capabilities of the WCR method. For detailed information, please refer to Appendix \ref{app.couple}.

\begin{figure}[htp]
\centering
\subfigure[]{\includegraphics[width=0.49\textwidth]{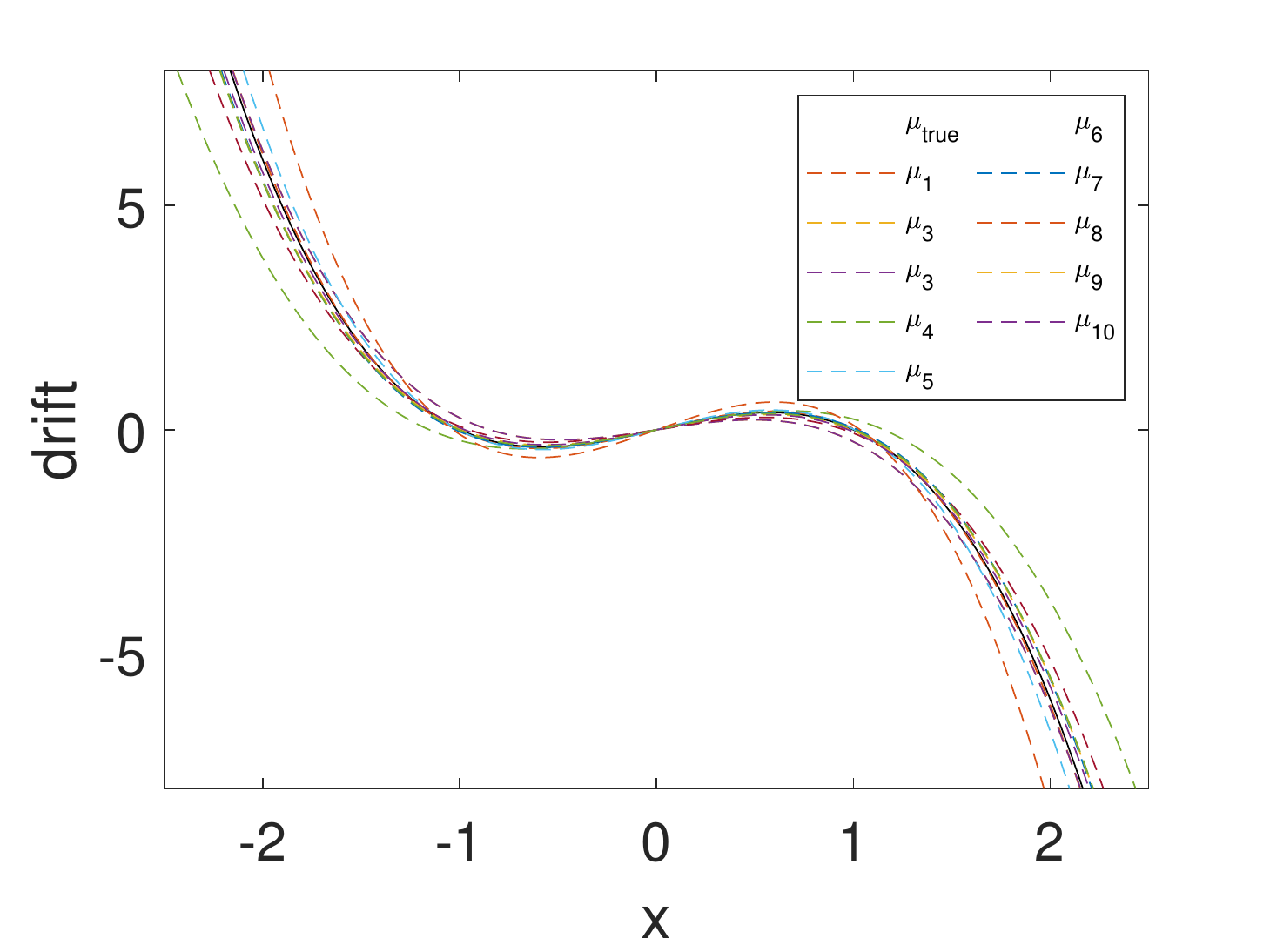}}
\subfigure[]{\includegraphics[width=0.49\textwidth]{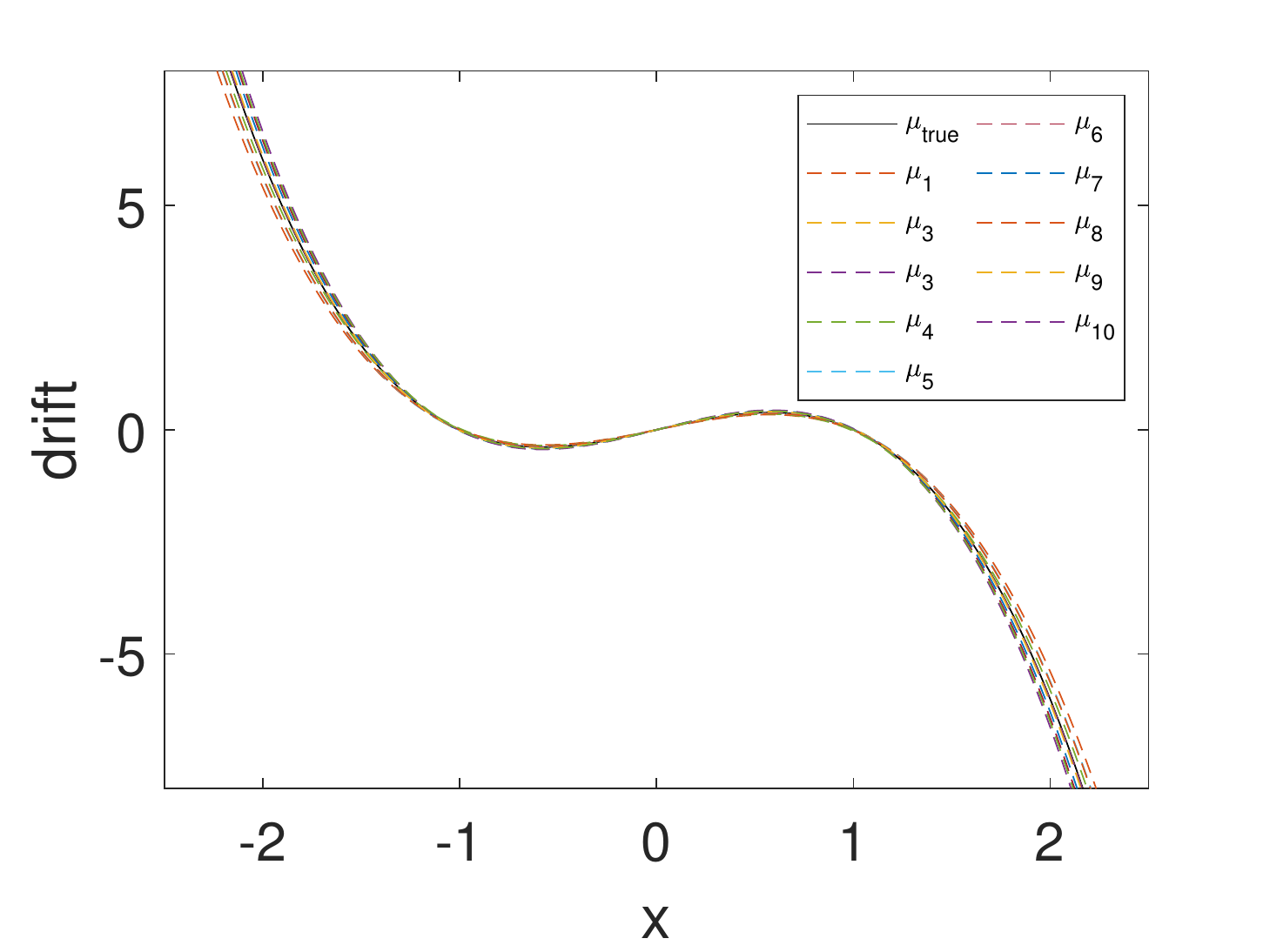}}
\subfigure[]{\includegraphics[width=0.49\textwidth]{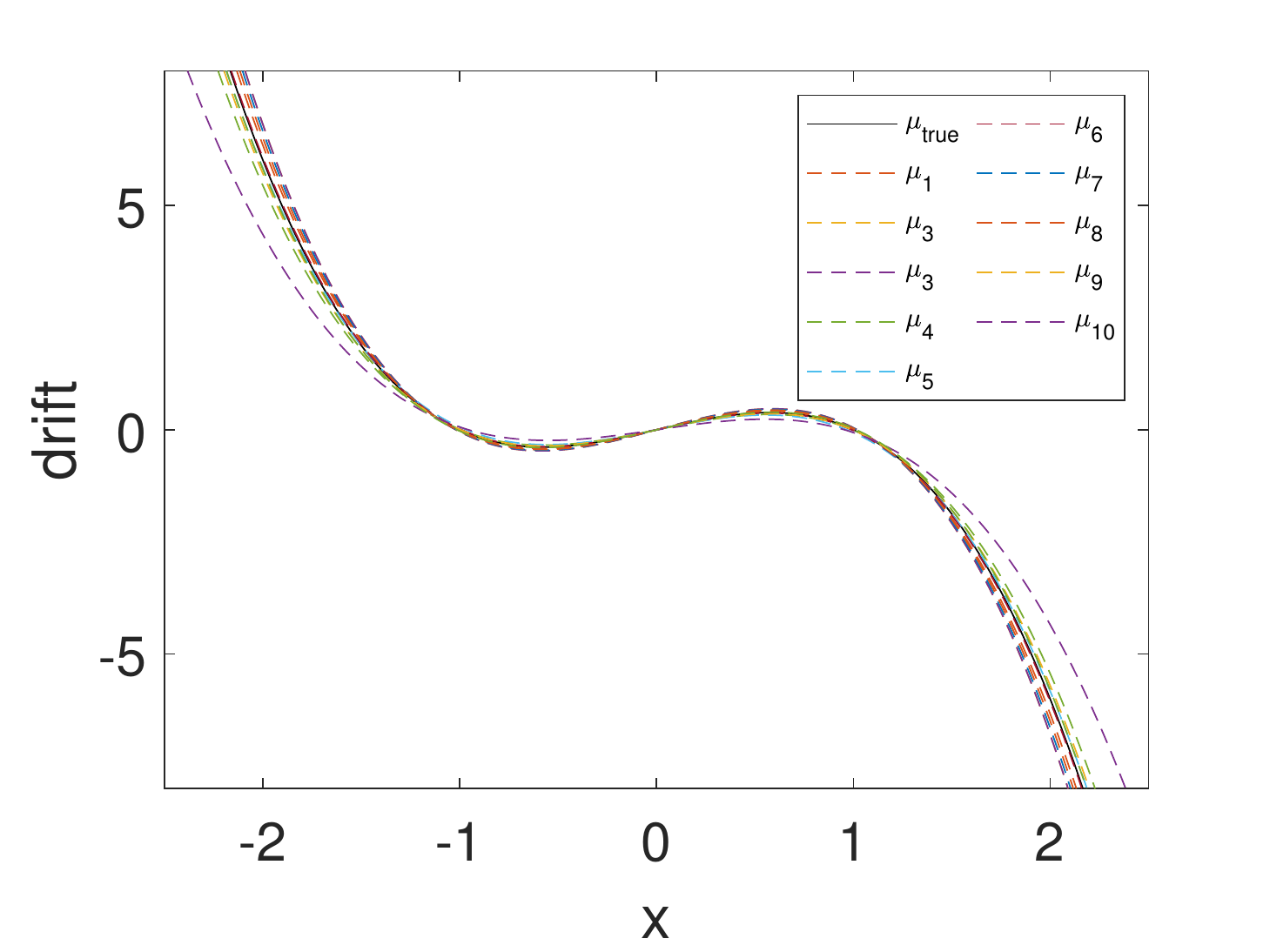}}
\subfigure[]{\includegraphics[width=0.49\textwidth]{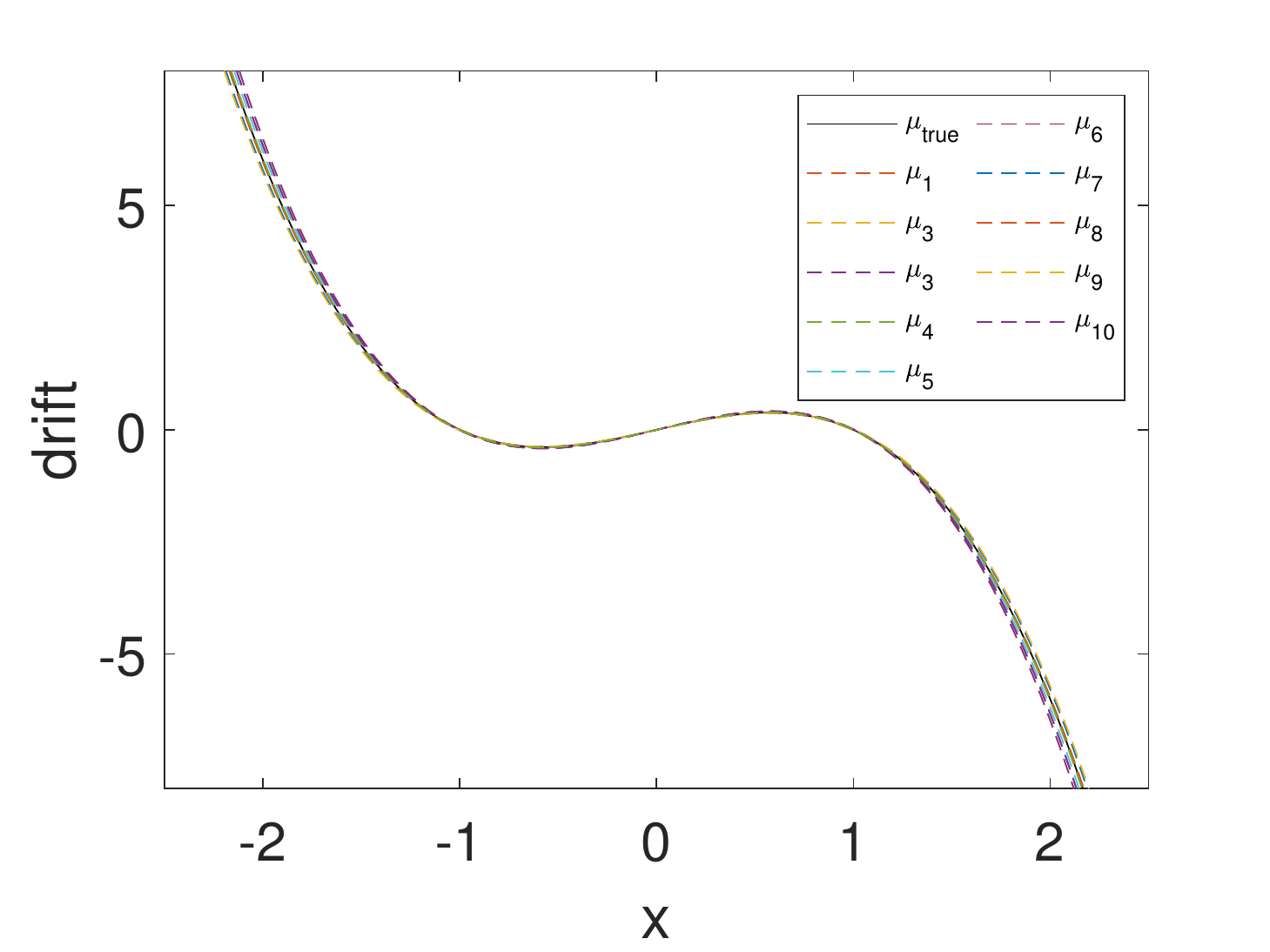}}
\caption{The results of 10d problems with different number of samples per time snapshot and gaussian kernels. Reveal the unkown drift and diffusion terms with (a) 10000 samples of $X_t$ using 1000 gaussian kernels; (b) 100000 samples of $X_t$ using 1000 gaussian kernels; (c) 10000 samples of $X_t$ using 10000 gaussian kernels; (d) 100000 samples of $X_t$ using 10000 gaussian kernels at time snapshots  $t = 0, 0.1, 0.2, 0.3, 0.4, 0.5, 0.6, 0.7, 0.8, 0.9, 1$ of 10-dimensional problem, where the samples are generated by the given SDE with drift term $\mu\_\text{true} = x-x^3$ and diffusion term $\sigma=1$ in per dimension. The inference results are denoted by $\mu_i$ for the learned drift in the $i$-th dimension.} \label{fig.10d}
\end{figure}

In the 20-dimensional case, we take 10000 gaussian kernels and  100000 samples of $X_t$. The experiment is completed within 47 minutes on the MacBook Pro and the result is shown in Figure \ref{fig.20d}

\begin{figure}[htp]
\centering
\includegraphics[width=0.8\textwidth]{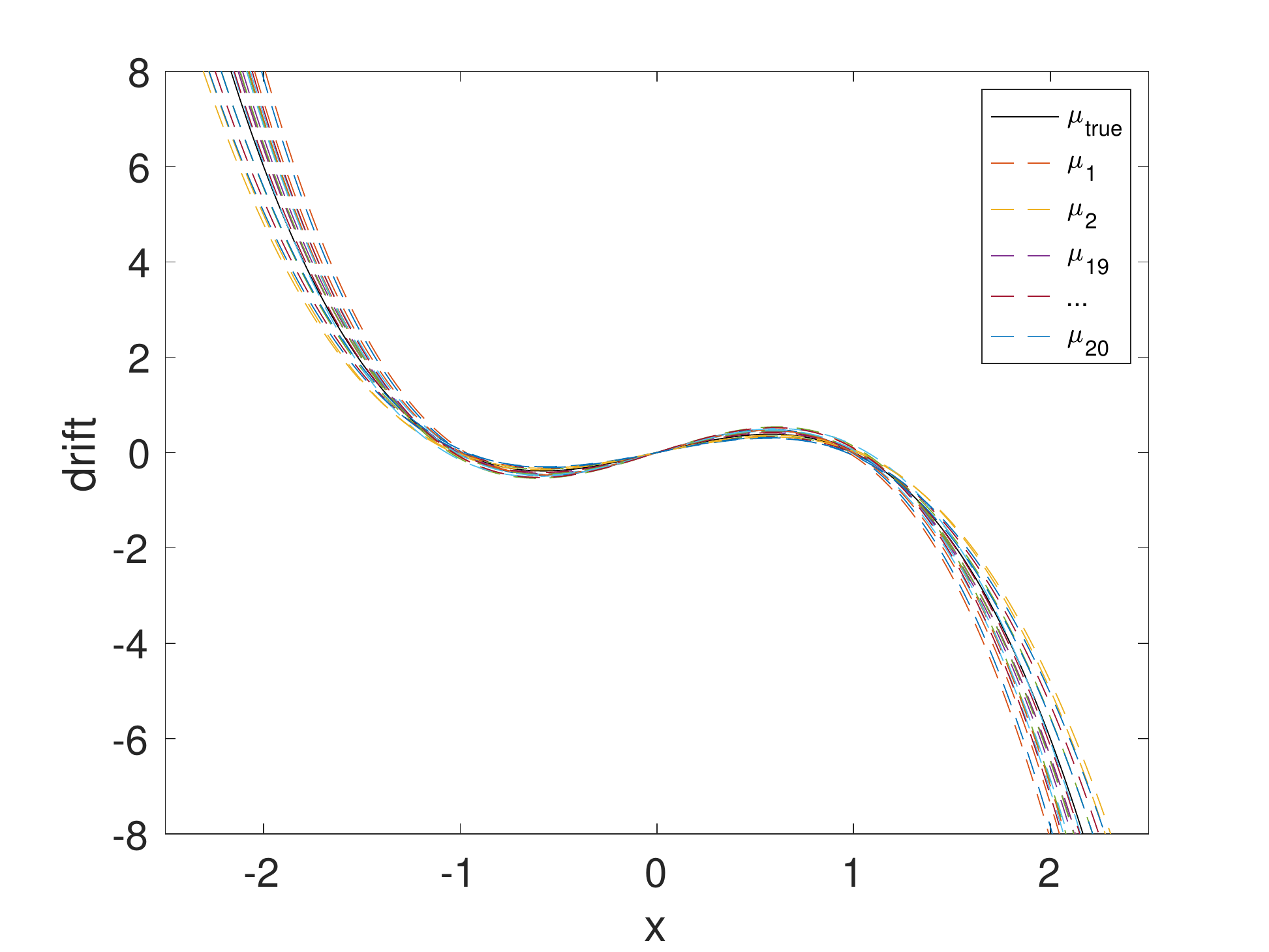}
\caption{The results of 20-dimensional problems with 100000 samples of $X_t$ using 10000 gaussian kernels at time  $t = 0, 0.1, 0.2, 0.3, 0.4, 0.5, 0.6, 0.7, 0.8, 0.9, 1$, where the samples are generated by the given SDE with drift term $\mu\_\text{true} = x_i-x_i^3$ and diffusion term $\sigma=1$ in per dimension. The inference results are denoted  as $\mu_i$ for the $i$-th dimensional drift.} 
\label{fig.20d}
\end{figure}


\subsection{Performance on data of different qualities} \label{sec.performance}

In real applications, the data obtained by various means usually contain many missing points and even flaws. The data may be also noisy with useful signals deeply buried. These facts make analyzing and extracting useful models from the data hard and tricky. We have shown that WCR performs well with only several snapshots with non-equally spaced data in the one dimensional case. In this subsection, we further investigate the performance of WCR on the data with different qualities, such as short and long time intervals, different number of the samples in the snapshots, and noises.

The true model in this experiment has the following form 
\begin{equation}
\begin{split}
dx_1 &= -0.5 x_1 dt + dW_t\\
dx_2 &= -0.7 x_2 dt + dW_t\\
dx_3 &= -x_3 dt + dW_t,
\end{split}
\end{equation}
where the three dimensions are not coupled for the simplicity and the mainly focus on the data qualities. The raw data are generated by integrating the true model
from $t=0$ to $t=10$, and different kinds of snapshots are chosen as the aggregate data. Noise is added for the study of the noise affect. 

To reveal the hidden dynamics, the unknown drift term is expanded by the first order complete polynomials as 
\begin{equation}
    \mbox{drift} = 
    \begin{pmatrix}
        \theta_{00} + \theta_{01}x_1 + \theta_{02}x_2 + \theta_{03}x_3 \\
        \theta_{10} + \theta_{11}x_1 + \theta_{12}x_2 +\theta_{13}x_3 \\
        \theta_{20} + \theta_{11}x_1 +\theta_{22}x_2 + \theta_{23}x_3
    \end{pmatrix},
\end{equation}
and the diffusion matrix is approximated by diagonal matrix 
$
D = \begin{pmatrix}
\sigma_1 &&\\
& \sigma_2&\\
&&\sigma_3\\
\end{pmatrix}
$,
where $[\theta_{ij}]$ and $[\sigma_i]$ are tunable parameters to be revealed.  

\textbf{Time interval and sample number.}  We investigate in this part how the time interval and the sample number of the data affect the accuracy of the results in our WCR method.  We consider four cases of the time interval $\Delta t_1 = 0.1$, $\Delta t_1 = 0.2$, $\Delta t_1 = 0.5$ and $\Delta t_1 = 0.9$ of the snapshots in the total interval $[0, 10]$.  To be concise, the total number of the snapshots $L=\frac{T}{\Delta t} +1$, where $T=10$ and $\Delta t$ is one of the time interval. For example we have $L=101$ snapshots in the case of $\Delta t_1=0.1$ with each snapshot are equally spaced with $\Delta t_1=0.1$. For the number of the samples in each snapshot, we also consider four cases with $N_1=1000$, $N_2=2000$, $N_3=5000$, and $N_4=10000$. Thus we have totally tested $4\times 4=16$ cases to investigate the performance of WCR on the time interval and sample number of the data. 

In all experiments, other setups are kept the same where 100 Gaussian functions with standard variance $\gamma=1$ are sampled as the test functions. And the Milne method is used for the temporal derivatives. Three different random seeds are used in each experiment for the average result. All experiments share the same three seeds for the fair comparison. 

We summarize all the results in Figure \ref{fig.3d.linear}. With only $N=1000$ samples and mild time interval $\Delta t=0.9$, the maximum relative error can reach a good accuracy around $5\%$. Further, with the increasing number of the samples in each snapshot and the decreasing time interval from $\Delta t_4=0.9$ to $\Delta_1=0.1$, from $N_1=1000$ to $N_2=10000$, the results generally get better. And all of the cases result in a low Max Relative Error giving the proof of the robustness of our method in relative poor data qualities. 

\begin{figure}[htp]
\centering
\subfigure[]{\includegraphics[width=0.38\textwidth]{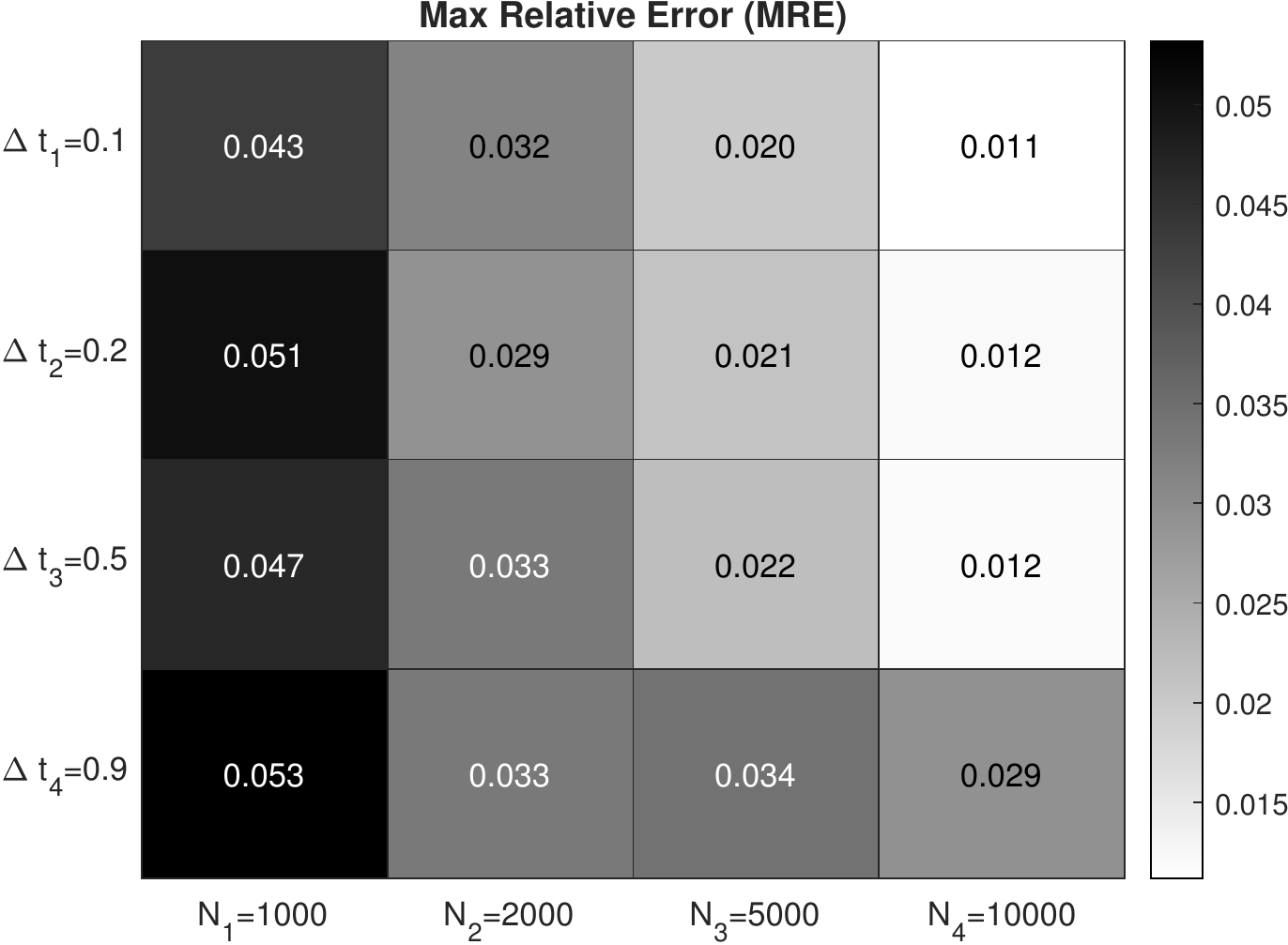}}\hspace{2em}
\subfigure[]{\includegraphics[width=0.5\textwidth]{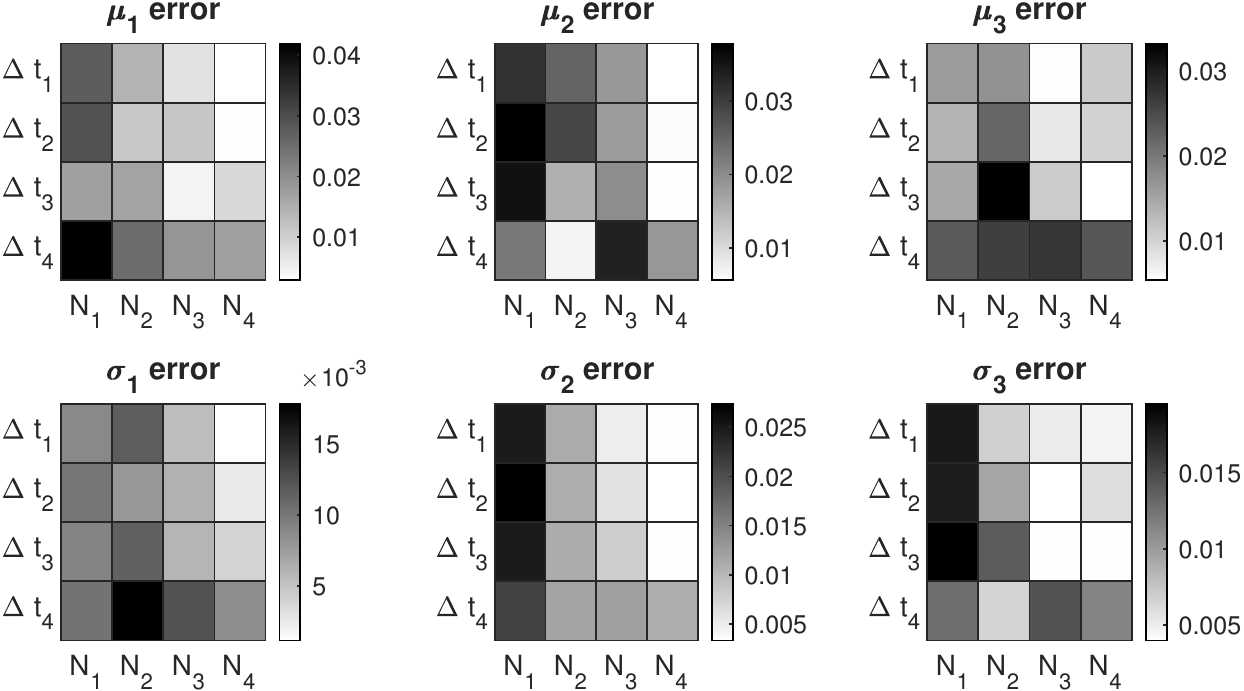}}
\caption{The heat map of the errors of the drift coefficients and diffusion for the 3-dimensional stochastic process in $T=[0, 10]$ with different trajectory samples $N_1=1000$, $N_2=2000$, $N_3=5000$, $N_4=10000$ and different time interval $\Delta t_1=0.1$, $\Delta t_2=0.2$, $\Delta t_3=0.5$, and $\Delta t_4=0.9$. The true governing stochastic equation reads $dx_1 = -0.5x_1dt + dW_1$, $dx_2 = -0.7x_2dt + dW_2$, and $dx_3 = -x_3dt + dW_3$, where the coefficients of $\mu_1$, $\mu_2$, and $\mu_3$ are $-0.5$, $-0.7$ and $-1$, and the diffusion $\sigma_1=\sigma_2=\sigma_3=1$. (a) shows the max relative error of the coefficients for $N_i$ and $\Delta t_j$; (b) shows the relative error of per coefficient for $N_i$ and $\Delta t_j$ with $i,j=1,2,3,4$.
} \label{fig.3d.linear}
\end{figure}

\textbf{Noise.}  In this part, we consider the performance of our WCR method on the noisy data. The raw data are obtained by the same procedure as the above and  the random noise is added with the noise level $\delta$
\begin{equation*}
\hat{\bf{x}}_i^j = \bf{x}_{i}^j + \delta\mathcal{U}_i^j\bf{x}_{i}^j,
\end{equation*}
where $\mathcal{U}_i^j$ is an uniform random variable in $[-1, 1]$. 
Here the data with 10000 samples in each snapshot are used and time intervals $\Delta t=0.1$ are chosen to show the robustness of our methods on the noise. 

The results are listed in Table \ref{tab.3d.linear} and the Max Relative Error reaches around 0.6\% with the white noise level $\delta=10\%$ for both two cases. Further, when we dive into the detail of the accuracy for drift and diffusion terms respectively, adding noise to the data didn't change the accuracy of the diffusion term much unlike the drift term.

\begin{table}[htp]
    \centering
    \begin{tabular}{cccccc}
         \toprule
         Noise & $\delta=0\%$ & $\delta=10\%$ & $\delta=20\%$ & $\delta=30\%$ & $\delta=40\%$ \\
         \midrule
         MRE & 0.76\% & 0.63\% & 1.92\% & 5.62\% & 11.69\% \\
         MRE in drift& 0.49\% & 0.54\% & 1.92\% & 5.62\% & 11.69\% \\
         MRE in diffusion & 0.76\% & 0.63\% & 0.28\% & 0.94\% & 2.36\% \\
         \bottomrule
    \end{tabular}
    \caption{The results of Three-dimensional linear problem with different time snapshot and noise. Reveal the unknown drift and diffusion terms with $\Delta t=0.1$ and $\delta=0,10,20,30,40\%$, where $10000$ samples each snapshot are generated by the given SDE with drift term $\mu(x,y,z)=[-0.5x,-0.7y,-z]^T$ and diffusion term $\sigma=I_{3\times 3}$.}
    \label{tab.3d.linear}
\end{table}

\section{Conclusion}

In this work, leveraging the weak form of the Fokker-Planck equation and the collocations of the Gaussian kernels, we proposed a framework called the Weak Collocation Regression method (WCR) to fast reveal the hidden stochastic dynamics from high-dimensional aggregate data. The lack of trajectory information makes aggregate data (unpaired data) more difficult to reveal the hidden dynamics. However, the data distribution follows the Fokker-Planck equation under some assumptions, such as the Brownian motion. By transferring the spatial derivatives to the test function in the weak form of the Fokker-Planck equation, we have an integral form of the density function with derivatives of the test functions. Thus the integral form can be easily approximated by summing the derivatives of the test functions over the samples. Using collocations of the Gaussian functions as the test functions and the dictionary representation of the unknown terms, we build an extensive linear system. Linear regression and sparse identification lead to unknown terms, thus revealing the hidden dynamics.

Our experiments show that WCR is numerically efficient, taking only $0.02s$ on the Macbook Pro for the $1$-dimensional problem with remarkable accuracy and only seconds for $3$-dimensional and $4$-dimensional problems. With the LMMs of variable step size, WCR achieves good accuracy in the non-equally spaced time data. Approximating the integral form using the summations over samples makes the WCR method naturally relieve the curse of dimensionality. And the extra computational cost is required on the collocations of the kernels but does not exponentially increase with dimension. The variable-dependent diffusion problem and coupled drift terms usually bring more difficulties, but the WCR method exhibits high accuracy in these complex problems. And for the data with noise, the WCR method also obtains a stable performance when adding the white noise. With different time intervals and samples, WCR shows its robustness without losing too much accuracy when the data quality worsens.

When comparing with the methods presented in the literature, \cite{ma2021learning} employed SDE to generate data at encrypted time points and optimized it using the Wasserstein distance. Similarly, \cite{yang2022generative} used SDE to artificially augment data and compared its distance with existing data using GANs. On the other hand, \cite{chen2021solving} incorporated the residual of the FP equation across the entire space and time into the objective function, computing the derivatives through auto-differentiation. In contrast, our methods directly solve the algebraic equations of unknown terms using the available data. The accuracy of the results obtained greatly relies on the precision of the derivatives approximated by the data. However, by utilizing Gaussian functions as test functions in the weak form, we significantly alleviate this issue as it eliminates the need for spatial derivatives. Consequently, the only remaining task is to accurately approximate the temporal derivatives of the one-dimensional sequence, as shown in equation \eqref{eq.monte}. It is important to note that the accuracy still depends on the precision of the temporal derivatives, necessitating relatively dense snapshots over time.

Despite the advantages mentioned above, the Weak Collocation Regression method still has some limitations for investigation in future work. Our approach slows down for the ultimately high-dimensional problems such as 100 dimensions since more Gaussian kernel collocations are needed, requiring high computational costs. More effective sampling schemes, such as active learning for the collocation of the Gaussian kernels, are necessary to reduce the computational cost. Neural networks might be another choice for the test function with a min-max optimization for reducing the cost. If prior information about the problem is available, it can be effectively incorporated into the dictionary, thereby mitigating the curse of dimensionality. However, if we lack prior information regarding the high-dimensional dynamics, expanding the unknown terms in a general basis form would lead to a significant increase in the number of terms. Consequently, traditional linear regression may prove less efficient. To address this issue, it may be worth considering a neural network approximation for the unknown terms. By adopting this approach, our framework remains intact, albeit with the substitution of linear regression for machine learning optimization schemes like gradient descent. However, this particular modification is left as an avenue for future research and exploration. Further discussion about the general form besides the Brownian motion is needed but also left for future work.

\section*{Acknowledgements}
This work was supported by the National Key R\&D Program of China (Grant No. 2021YFA0719200) and the National Natural Science Foundation of China (Grant No. 92370125). The authors would like to thank the helpful discussions from Dr. Liu Hong.

\bibliography{SDE_kernel.bib}
\bibliographystyle{unsrt}

\appendix

\section*{Appendix}

\section{Error Analysis}\label{app.error}

In this section, we are going to analyze the error of the Weak Collocation Regression (WCR) method. The key ingradients in our error analysis are inspired by \cite{messenger2022learning, messenger2022asymptotic}. The error in our method mainly consists of the following aspects: (a) approximations of the integrals; (b) approximations of the temporal derivative; (c) the linear regression. Briefly speaking, the error in computing the integrals by Monte Carlo method is $\mathcal{O}(\frac{1}{\sqrt{N}})$, and the error in approximating the temporal derivative by the linear multistep method is $\mathcal{O}(\Delta t^\alpha)$. In this section, our primary focus lies in examining the influence of sample size $N$ and the time interval $\Delta t$ on the error. The order of the linear multistep method is denoted as $\alpha$, and the ordinary least squares method is used.

In a general context, we represent the linear multistep method for the differential equation $\frac{d}{dt}\boldsymbol y(t)=\boldsymbol f(t)$ as $\boldsymbol y_{n+k} = \sum_{i = 0}^{k-1}\alpha_i \boldsymbol y_{n+i}+\Delta t \sum_{i = 0}^k \beta_i \boldsymbol f_{n+i}$. Consequently, we define the temporal discrete operator for the left-hand side as $\mathcal{D}_t^L \boldsymbol y(t):= \frac{1}{\Delta t}(\boldsymbol y(t+k\Delta t) - \sum_{i=0}^{k-1}\alpha_i \boldsymbol y(t+i\Delta t))$, while the right-hand side is represented as $\mathcal{D}_t^R \boldsymbol f(t,y):= \sum_{i = 0}^k  \beta_i \boldsymbol f(t+i\Delta t,y(t+i\Delta t))$. Once the library $\Lambda=\{\Lambda_1(\boldsymbol x), \Lambda_2(\boldsymbol x),\cdots,\Lambda_b(\boldsymbol x)\}$ and the test functions $\{\psi_m\}_{m=1}^M$ are fixed, the true value $\overline{\Abf}\in\mathbb{R}^{M(L-k)\times(db+d^2b)}$ of the coefficient matrix $\Abf\in\mathbb{R}^{M(L-k)\times(db+d^2b)}$ is determined as follows: 
\begin{equation}
    \overline{\Abf}_{m}^{ij} = \left\{
    \begin{aligned}
    & \sum_{l=0}^k \beta_l \int_{\mathbb{R}^d} p(t_i,\boldsymbol x)\Lambda_j(\boldsymbol x)\cdot \nabla\psi_m(\boldsymbol x) d\boldsymbol x, \ 1\leq j\leq db \\
    & \sum_{l=0}^k \beta_l \int_{\mathbb{R}^d} \sum_{r,s=1}^d p(t_i,\boldsymbol x)\Lambda_j(\boldsymbol x)\frac{\partial}{\partial x_rx_s}\psi_m(\boldsymbol x) d\boldsymbol x, \ db+1\leq j\leq db+d^2b
    \end{aligned}\right.,
    \overline{\Abf} = 
    \begin{pmatrix}
       \overline{\Abf}_1 \\ \overline{\Abf}_2 \\ \cdots \\ \overline{\Abf}_M  
    \end{pmatrix},
\end{equation}
where $\overline{\Abf}_m:=(\overline{\Abf}_m^{ij})\in\mathbb{R}^{(L-k)\times(db+d^2b)}$ is the block within $\overline{\Abf}$ derived from test function $\psi_m$, and the operation $\sum_{l=0}^k \beta_l\times(\cdot)$ represents the linear multistep method. Importantly, it's noteworthy that the true value $\overline{\Abf}$ remains unaffected by both sample number $N$ and time interval $\Delta t$, thus establishing its independence from the chosen computational format. Subsequently, we list the assumptions for the problem.

\begin{assumption}\label{assump1}
\hrule
\begin{enumerate}
\item The aggregate data $\mathbb{X} = \{\{\bx^i_{t_j}\}_{i=1}^{N}\}_{j=1}^L$ is observed from a strong solution to \eqref{eq.sde} for $t_j\in[0,T]$ with constant sample number.
\item The probability density function $p(x,t)$ is smooth enough in the time direction. $\{\frac{\partial^i}{\partial t^i}p(x,t)\}_{i=1}^{\alpha+1}$ is uniformly bounded with respect to $x$.
\item The test functions $\{\psi_k\}_{k=1}^M$ belong to the Schwartz space 
$$ \mathcal{S}(\mathbb{R}^d):=\{f\in C^\infty(\mathbb{R}^d)|\sup_{x\in\mathbb{R}^d}|x^\alpha D^\beta f(x)|<\infty, \forall \alpha, \beta\} $$
The functions in library $\Lambda$ have algebraic growth. The test functions together with the library $\Lambda$ are such that the true value matrix $\overline{\Abf}$ has full column rank\footnote{It is challenging to guarantee the existence of a library in theory such that the true value matrix $\overline{\Abf}$ has the full column rank. In practice, the number of rows of $\overline{\Abf}$ is much greater than the number of its columns. And we indeed observe this feature and that $\overline{\Abf}$ being with the full column rank consistently in all our experiments. The interpretation regarding this particular assumption can also be found in Remark 2.4 of \cite{messenger2022asymptotic}.}. There exist a universal bound $C_1$ such that $\forall \psi(\boldsymbol x)\in \mathbb C_{\rho,\gamma}^d ,\int_{\mathbb{R}^d}|\psi(\boldsymbol x)|dx\leq C_1$. 
\item Moreover, $\nrm{\overline{\Abf}^\dagger}_\infty\leq C_{2}$ almost surely, where $\nrm{\overline{\Abf}^\dagger}_\infty$ is the induced matrix $\infty$-norm of $\overline{\Abf}^\dagger$.
\item The true functions $\boldsymbol \mu^\star$ and $D^\star$ are in the span of $\Lambda$.

\end{enumerate}
\hrule
\end{assumption}

The coefficient matrix $\Abf\in\mathbb{R}^{M(L-k)\times(db+d^2b)}$ used in linear regression is provided by the following expression:
\begin{equation}
    \Abf^{ij}_m = \left\{
    \begin{aligned}
    & \sum_{l=0}^k \beta_l\left(\frac{1}{N}\sum_{n=1}^N \Lambda_j(x_{t_i}^n) \cdot \nabla\psi_m(x_{t_i}^n)\right), \ 1\leq j\leq db \\
    & \sum_{l=0}^k \beta_l\left(\frac{1}{N}\sum_{n=1}^N \sum_{r,s=1}^d \Lambda_j(x_{t_i}^n)\frac{\partial}{\partial x_rx_s}\psi_m(x_{t_i}^n)\right), \ db+1\leq j\leq db+d^2b
    \end{aligned}\right.,
    \Abf = 
    \begin{pmatrix}
       \Abf_1 \\ \Abf_2 \\ \cdots \\ \Abf_M  
    \end{pmatrix},
\end{equation}
where $\Abf_m:=(\Abf_m^{ij})\in\mathbb{R}^{(L-k)\times(db+d^2b)}$ is the block within $\Abf$ derived from test function $\psi_m$, and $x_{t_i}^n$ represents the n-th sample at time $t_i$. Here the coefficient matrix $\Abf$ exhibits a relationship with sample number $N$, but it remains independent of the time interval $\Delta t$. It becomes evident that the size of matrix $\Abf$ and $\overline{\Abf}$ are congruent, with the difference in corresponding elements quantified as $\overline{\Abf}^{ij}-\Abf^{ij}=\mathcal{O}(\frac{1}{\sqrt{N}})$. Subsequently, relying on assumption \ref{assump1}(3), it follows that $\Abf$ also attains full column rank for sufficiently large values of $N$. Likewise, in the ordinary least squares method, $\Abf^\dagger:=(\Abf^T\Abf)^{-1}\Abf^T$, which also implies $\|\Abf^\dagger\|_\infty=\| \overline{\Abf}^\dagger \|_\infty+\mathcal{O}(\frac{1}{\sqrt{N}})$. Therefore, in accordance with assumption \ref{assump1}(4), we can deduce that $\nrm{\Abf^\dagger}_\infty\leq 2C_{2}$ for sufficiently large values of $N$.

Given a test function $\psi(\boldsymbol x)$ and any probability density function $\rho(\boldsymbol x ,t)$, define the continuous-time weak-form residual
\begin{equation}
\begin{split}
\mathcal R(\rho(\boldsymbol x,t),\psi(\boldsymbol x)) =\int_{\mathbb{R}^d} \frac{\partial}{\partial t}\rho(\boldsymbol x,t)\psi(\boldsymbol x)d\boldsymbol x
- \int_{\mathbb{R}^d} \rho(\boldsymbol x,t)\bf{\mu}(\boldsymbol x)\cdot \nabla \psi(\boldsymbol x)d\boldsymbol x -\int_{\mathbb{R}^d} \rho(\boldsymbol x,t)\sum_{rs}^dD_{sr}\partial_{rs}\psi(\boldsymbol x)d\boldsymbol x,
\end{split}
\end{equation}
If $\rho(x,t) = p(x,t)$ is a weak solution to \eqref{eq.fp} and then the residual $\mathcal{R}(\rho(\boldsymbol x,t),\psi(\boldsymbol x))=0$ for any $t$ and test function $\psi$.To incorporate discrete effects, we introduce the discrete-time weak-form residual $\mathcal R_{\Delta t}(\rho(\boldsymbol x,t),\psi(\boldsymbol x))$ by replacing the differential operator with the LMM discret 
\begin{equation}
\begin{split}
\mathcal R_{\Delta t}(\rho(\boldsymbol x,t),\psi(\boldsymbol x)) =\mathcal{D}_t^L \int_{\mathbb{R}^d}   \rho(\boldsymbol x,t)\psi(\boldsymbol x)d\boldsymbol x
- \mathcal{D}_t^R[\int_{\mathbb{R}^d} \rho(\boldsymbol x,t)\bf{\mu}(\boldsymbol x)\cdot \nabla \psi(\boldsymbol x)d\boldsymbol x + \rho(\boldsymbol x,t)\sum_{rs}^dD_{sr}\partial_{rs}\psi(\boldsymbol x)d\boldsymbol x],
\end{split}
\end{equation}
Next, we estimate the errors induced by the discrete-time derivative and Monte Carlo computation of expectations respectively. And the final error analysis is stated in Theorem \ref{thm:w} by incorporating the linear regression.

\vspace{1em}
\textbf{Step1: Estimate the error of linear multistep method}

We first estimate the approximation error of discrete time derivative, which we denoted as $\mathcal E_1(\rho(\boldsymbol x,t),\psi(\boldsymbol x))$ for convenience  
\begin{equation}
\begin{split}
\mathcal E_1(\rho(\boldsymbol x,t),\psi(\boldsymbol x))=\mathcal R(\rho(\boldsymbol x,t),\psi(\boldsymbol x))-\mathcal R_{\Delta t}(\rho(\boldsymbol x,t),\psi(\boldsymbol x)).
\end{split}
\end{equation}
Using assumption \ref{thm:w}(2) we can deduce the truncation error
\begin{equation}
    \vert\partial_t p(\boldsymbol x,t) - \mathcal{D}_t^L p(\boldsymbol x,t) +(\mathcal{D}_t^R  p(\boldsymbol x,t)- p(\boldsymbol x,t) )\bf{\mu}(\boldsymbol x)\cdot \nabla \psi(\boldsymbol x) +(\mathcal{D}_t^R  p(\boldsymbol x,t)- p(\boldsymbol x,t) )\sum_{rs}^dD_{sr}\partial_{rs}\psi(\boldsymbol x)|\leq C_3 \Delta t^\alpha
\end{equation}
where $C_3$ is independent of $\Delta t$, and directly implies the order of approximation error 

\[
\begin{aligned}
\left\vert
\mathcal E_1(p(\boldsymbol x,t),\psi(\boldsymbol x))
\right\vert& \leq 
\vert\int_{\mathbb{R}^d}|\psi(\boldsymbol x)| \cdot\vert\partial_t p(\boldsymbol x,t) - \mathcal{D}_t^L p(\boldsymbol x,t) \\
&+(\mathcal{D}_t^R  p(\boldsymbol x,t)- p(\boldsymbol x,t) )\bf{\mu}(\boldsymbol x)\cdot \nabla \psi(\boldsymbol x) \\
& +(\mathcal{D}_t^R  p(\boldsymbol x,t)- p(\boldsymbol x,t) )\sum_{rs}^dD_{sr}\partial_{rs}\psi(\boldsymbol x)| d\boldsymbol x\\
&\leq\int_{\mathbb{R}^d}|\psi(\boldsymbol x)| \cdot C_3\Delta t^\alpha d\boldsymbol x \leq C_{4}\Delta t^\alpha
\end{aligned}
\] 
Since $\psi$ belongs to the Schwartz space , the spatial integration is bounded. Here $C_4$ is dependent of the test function $\psi$, but independent of $\Delta t$.

\vspace{1em}
\textbf{Step2: Estimate the error of Monte Carlo method}

Then we consider the difference between exact probability density function $p(x,t)$ and empirical density function $p^N(x,t)$ with discrete temporal differential operator. Similarly, the approximation error is denoted as $\mathcal E_2(p(\boldsymbol x,t),p(\boldsymbol x))$
\begin{equation}
\begin{split}
\mathcal E_2(\rho(\boldsymbol x,t),\psi(\boldsymbol x))=\mathcal R_{\Delta t}(\rho(\boldsymbol x,t),\psi(\boldsymbol x))-\mathcal R_{\Delta t}(\rho^N(\boldsymbol x,t),\psi(\boldsymbol x)).
\end{split}
\end{equation}


Then
\[
\begin{aligned}
\left\vert
\mathcal E_2(p(\boldsymbol x,t),\psi(\boldsymbol x))
\right\vert& \leq 
\vert\int_{\mathbb{R}^d}|\psi(\boldsymbol x)|\cdot \vert \mathcal{D}_t^L p(\boldsymbol x,t) - \mathcal{D}_t^L p^N(\boldsymbol x,t) \\
&+\mathcal{D}_t^R  \left(p^N\left(\boldsymbol x,t\right)-p\left(\boldsymbol x,t\right)\right) \bf{\mu}(\boldsymbol x)\cdot \nabla \psi(\boldsymbol x) +\mathcal{D}_t^R\left(p^N\left(\boldsymbol x,t\right)-p\left(\boldsymbol x,t\right)\right)\sum_{rs}^dD_{sr}\partial_{rs}\psi(\boldsymbol x) 
\vert  d\boldsymbol x \\
\end{aligned}
\]

which can be decomposed of two type of term. The first type is the inner product of test function and probability density function with discrete temporal differential operator.  Hence, by setting $\alpha_k =1$, the first term can be estimated by 
\[
\begin{aligned}
&|\int_{\mathbb{R}^d}\psi(\boldsymbol x)\mathcal{D}_t^L p(\boldsymbol x,t)  
-\psi(\boldsymbol x)\mathcal{D}_t^L p^N(\boldsymbol x,t) d\boldsymbol x|\\
\leq& \sum_{i = 0}^k\left\vert \alpha_i\right\vert \left\vert\int_{\Omega }\psi(\boldsymbol x)\frac{p(\boldsymbol x,t+i\Delta t) -p^N (\boldsymbol x,t+i\Delta t)}{\Delta t}  d\boldsymbol x\right\vert \triangleq I_1,
\end{aligned}
\]
by the result of Theorem 1.1 of \cite{bolley2011stochastic} ,we have
\[
    \Ebb \left\vert\int_{\Omega }\psi(\boldsymbol x)\left(p(\boldsymbol x,t) -p^N (\boldsymbol x,t)\right) d\boldsymbol x\right\vert \leq \frac{C_5}{\sqrt{N}}
\]
with $C_5$ depending on test function $\psi$, but independent with $N,t,\Delta t$. So
\[
\Ebb [I_1] \leq \frac{C_{6}}{\sqrt{N}\Delta t}.
\]
Here $C_6$ is independent of $\Delta t,N$, and dependent of test function $\psi$. The other terms are of the form $|\int_{\mathbb{R}^d}\Phi(\boldsymbol x,t)(p(\boldsymbol x,t) - p^N(\boldsymbol x,t))d\boldsymbol x|$, and thus the expectation is $\mathcal O(\frac{1}{\sqrt{N}})$.

After further accounting for the error from linear regression, Theorem \ref{thm:w} provides the final error estimation for this method.

\setcounter{theorem}{0}
\begin{theorem}\label{thm:w}
Let $\hat{\boldsymbol{\zeta}}=\Abf^\dagger\bbf\:=(\Abf^T\Abf)^{-1}\Abf^T\bbf$ be the learned model coefficients and $\boldsymbol\zeta^\star$ the true model coefficients. There exists $C$ independent of $N$ and $\Delta t$ such that
\[\Ebb\left[\nrm{\hat{\boldsymbol{\zeta}}-\boldsymbol\zeta^\star}_\infty\right] \leq C(\frac{1}{\sqrt{N}\Delta t}+\Delta t^\alpha).\]
\end{theorem}

\begin{proof}
Using that $\boldsymbol \mu^\star$and $D^\star$ are in the span of $\mathbb C_{\rho,\gamma}^d$, we have that 
\[\bbf_k =\mathcal{D}_t^L\int_{\mathbb{R}^d}\psi_k(\boldsymbol x) p^N(\boldsymbol x,t)dx   = \mathcal R_{\Delta t}(p^N(\boldsymbol x,t),\psi_k(\boldsymbol x)) + \Abf^T_k \boldsymbol \zeta^\star:=\Rbf_k +\Abf^T_k \boldsymbol \zeta^\star,\]
where $\Abf^T_k$ is the $k$th row of $\Abf$. From the previous result, we have

\[
\begin{aligned}
\Ebb\left[|\Rbf_k|\right]&\leq \Ebb\left[|\mathcal E_2(p(\boldsymbol x,t),\psi_k(\boldsymbol x))|\right]+ \Ebb\left[|\mathcal E_1(p(\boldsymbol x,t),\psi_k(\boldsymbol x))|\right] + \Ebb\left[|\mathcal R(p(\boldsymbol x,t),\psi_k(\boldsymbol x))| \right]\\
&\leq C'\left( \Delta t^\alpha+\frac{1}{\sqrt{N} \Delta t}\right),
\end{aligned}\]
and $C'$ is independent of $N$ and $\Delta t$. Based on the previously derived results that $\Abf$ has full column rank, it holds that $\hat{\boldsymbol{\zeta}}=\Abf^\dagger\bbf=\Abf^\dagger\Rbf +\boldsymbol\zeta^\star$, hence the result follows from the uniform bound on $\nrm{\Abf^\dagger}_\infty$:
\[\Ebb\left[\nrm{\hat{\boldsymbol{\zeta}}-\boldsymbol\zeta^\star}_\infty\right]\leq \nrm{\Abf^\dagger}_\infty\Ebb\left[\nrm{\Rbf}_\infty\right]\leq C'\nrm{\Abf^\dagger}_\infty \left( \Delta t^\alpha+\frac{1}{\sqrt N \Delta t}\right)\leq 2C'C_2 \left( \Delta t^\alpha+\frac{1}{\sqrt N \Delta t}\right).\]
\end{proof}

For fixed sample number $N$ , the observation time interval $\Delta t$ is optimally chosen as $N^{-\frac{1}{2(\alpha-1)}}$ so that the error is $\mathcal O(\frac{1}{\sqrt{N}} )$.

\section{Computational complexity analysis}\label{app.complexity}
In Weak Collocation Regression method, the explicit form of Gaussian function taken as test function and its first and second derivatives can be calculated in advance.
\begin{equation}
    \phi(\bx, \bf{\rho}, \bf{\gamma}) = \Pi_{i=1}^d \frac{1}{\gamma_i \sqrt{2 \pi}} e^{-\frac{1}{2}\left(\frac{x_i - \rho_i}{\gamma_i}\right)^{2}},
\end{equation}
\begin{equation} \label{eq.gauss1}
    \frac{\partial}{\partial x_i} \phi(\bx, \bf{\rho}, \bf{\gamma}) = -\phi(\bx, \bf{\rho}, \bf{\gamma}) \cdot \frac{x_i-\rho_i}{\gamma_i^2}
\end{equation}
\begin{equation} \label{eq.gauss2}
    \frac{\partial^2}{\partial x_i \partial x_j} \phi(\bx, \bf{\rho}, \bf{\gamma}) = 
    \begin{cases}
        \phi(\bx, \bf{\rho}, \bf{\gamma}) \left((\frac{x_i-\rho_i}{\gamma_i})^2 - \frac{1}{\gamma_i^2}\right) & i=j \\
        \phi(\bx, \bf{\rho}, \bf{\gamma})\cdot\frac{x_i-\rho_i}{\gamma_i^2}\cdot\frac{x_j-\rho_j}{\gamma_j^2} & i\neq j
    \end{cases}
\end{equation}
Therefore, the computational complexity of $\phi(\bx, \bf{\rho}, \bf{\gamma})$ is $\mathcal{O}(LNd)$ , where $L,N,d$ stands for number of snapshots, sample number and dimension respectively. Thanks to the object-oriented characteristics of python, it is not necessary to recompute  $\phi(\bx, \bf{\rho}, \bf{\gamma})$ in \eqref{eq.gauss1}\eqref{eq.gauss2}. 

On computing known coefficient matrix $B(\mathbb{X})$ and column vector $\mathbf{y}(\mathbb{X})$, there are $\mathcal{O}(LNMd)$ and $\mathcal{O}(LNM)$ computational cost respectively, for we expand the diagonal of the diffusion matrix with a fixed number of basis $b$. Here $M$ stands for the number of the Gaussian functions. When applying linear multistep method, we obtain the linear system $(A, \hat{\mathbf{y}})$ via $\mathcal{O}(L)$ additional computational cost.

On computing linear regression and sparse identification, \textit{np.linalg.lstsq} finishes the calculation with QR factorization. The computational complexity of QR factorization is $\mathcal{O}(b^2LM)$, which takes little time because $b$ is a constant, and it is already covered by $\mathcal{O}(LNMd)$ in computing coefficient matrix. A large number of experiments show that linear regression and sparse identification takes a very low proportion of the total execution time, see Table \ref{table.qr}. Therefore, the choice of computing matrix inverse has little effect.
\begin{table}[h]
    \centering
    \begin{tabular}{ccccccc}
        \toprule
        Time(s) & 1d cubic & 1d quintic & 2d sombrero & 3d cubic & 5d cubic & 10d cubic \\
        \midrule
        STRidge & 0.0004 & 0.001 & 0.132 & 0.0006 & 0.003 & 0.016  \\
        Total & 0.0177 & 0.541 & 28.9 & 3.84 & 66.5 & 213.5 \\
        STRidge/Total & 2.26\% & 0.18\% & 0.46\% & 0.02\% & 0.005\% & 0.008\% \\
        \bottomrule
    \end{tabular}
    \caption{STRidge and the total program time under different examples}
    \label{table.qr}
\end{table}

In conclusion, the computational complexity of Weak Collocation Regression method is $\mathcal{O}(LNMd)$ under condition of fixed the step of linear multistep method. It can be found that the 
computational complexity of WCR does not increase exponentially with the problem dimension.

\section{10-dimensional coupled system} \label{app.couple}
To illustrate our performance of WCR method in the high-dimensional coupled system, we consider a 10 dimensional aggregate data generated by the true model 
\begin{equation}
\begin{gathered}
    dX_t = \mu(X_t)dt + \sigma dW_t, \quad X_t\in\mathbb{R}^d \\ 
    \mu(x) = \begin{pmatrix} x_1-x_1^3 \\ x_2-x_2^3 \\ x_3-x_3^3 \\ \cdots \\ x_d - x_d^3 \end{pmatrix}, \quad 
    \sigma = 
    \begin{pmatrix}
        1 & 0 & 0 & 0 & \cdots & 0 \\ 
        1 & 1 & 0 & 0 & \cdots & 0 \\ 
        0 & 1 & 1 & 0 & \cdots & 0 \\ 
        0 & 0 & 1 & 1 & \cdots & 0 \\ 
        \vdots & \vdots & \vdots & \vdots & \ddots & \vdots \\
        0 & 0 & 0 & \cdots & 1 & 1
    \end{pmatrix}
\end{gathered}
\end{equation}
with $d=10$ and 100000 samples using Euler-Maruyama method. Several time snapshots of the data at time $t=0,0.1,0.2,0.3,\cdots,1$ are used as the experiment data. 

In the experiment, the forth order basis is used to approximate the drift term in each dimension as $\mu_i=\theta_0+\theta_1x_i+\theta_2x_i^2+\theta_3x_i^3$, and the diagonal and 1-subdiagonal in diffusion term is approximated by a tunable parameter $\sigma_i$, making totally $6d-1$ tunable parameters to be learned. Milne method is used for the temporal derivatives. 2000 and 10000 Gaussian functions were employed respectively, with a variance of 1 and mean values obtained by sampling from the region of the data using the Latin hypercube sampling method.

It is worth noting that the diffusion term has coupling, which means that the matrix $D=\frac{1}{2}\sigma\sigma^T$ we learn may not necessarily be a symmetric matrix. If we denote $\hat{D}=\frac{1}{2}(D+D^T)$, then $\hat{D}$ is a symmetrical matrix. By performing a Cholesky decomposition on $\hat{D}$, the lower triangular matrix obtained is the output diffusion term $\sigma$. Table \ref{tab.10d} presents the results of the calculation, which indicate that using 10,000 Gaussian functions, our method achieves a maximum relative error accuracy of 6.36\% within 30 minutes. The results here demonstrate that the WCR method is effective in addressing high-dimensional problems with coupling.
\begin{table}[h]
    \centering
    \begin{tabular}{ccccc}
        \toprule
        Gaussian & MRE in drift & MRE in diffusion & MRE & Time(min) \\ 
        \midrule
        $m=2000$ & 11.8\% & 6.22\% & 11.8\% & 5 \\ 
        $m=10000$ & 6.36\% & 4.35\% & 6.36\% & 30 \\ 
        \bottomrule
    \end{tabular}
    \caption{The results of 10-dimensional problem with coupled diffusion term. Reveal the unknown dynamics with $m=2000, 10000$ Gaussian functions, where there are 100000 samples each snapshots.}
    \label{tab.10d}
\end{table}

\section{STRidge Algorithm}

Here we provide the details of the sequential thresholded ridge regression (STRidge) algorithm. In the STRidge method, each linear regression step retains the variables that were not sparsified in the previous regression. And if the original linear equation has $n$ unknowns, the sparse regression operation is performed for a maximum of $n$ iterations. STRidge will terminate directly if either of the following two conditions is met: 1) After a regression step, no additional variables are removed compared to the previous regression; 2) All variables have been removed. For further details of the STRidge algorithm, please refer to Algorithm \ref{alg:STRidge}.

\begin{algorithm}[ht] 
    \caption{STRidge Algorithm for Solving Linear System $A\bm{x}=\bm{b}$} \label{alg:STRidge}
    \SetAlgoLined
    \KwResult{Vector $\bm{x} \in \mathbb{R}^n$ s.t. $A\bm{x} \approx \bm{b}$}
    \KwIn{Coefficient matrix $A \in \mathbb{R}^{m \times n}$, vector $b \in \mathbb{R}^m$, regular terms $\lambda > 0$, threshold $\eta>0$}
    Compute $x$ by ridge regression $x=\argmin_{w}\|A\bm{w}-\bm{b}\|^2+\lambda\|\bm{w}\|^2$, set $p=n$; \\ 
    \While{True}{
        Select index set $S^+=\{x>\eta\}$, $S^-=\{x\leq\eta\}$; \\ 
        \eIf{$\card\{S^+\}=p$}{break}{$p=\card\{S^+\}$}
        \If{$\card\{S^+\}=0$}{break}
        $x[S^-]=0, x[S^+]=\argmin_{w}\|A[:, S^+]\bm{w}-\bm{b}[:, S^+]\|^2+\lambda\|\bm{w}\|^2$
    }
    \If{$S^+\neq\emptyset$}{$x[S^+]=\argmin_{w}\|A[:, S^+]\bm{w}-\bm{b}[:, S^+]\|^2$}
\end{algorithm}

\end{document}